\documentclass[11pt]{article}

\usepackage[pagebackref,colorlinks=true,pdfpagemode=none,urlcolor=blue,
linkcolor=blue,citecolor=blue]{hyperref}

\usepackage{amsmath,amsfonts,amssymb,amsthm}
\usepackage{mathrsfs}
\usepackage{bm}
\usepackage{color}
\usepackage{algorithm}
\usepackage{algorithmic}

\usepackage{graphicx}
\usepackage{subfigure}
\usepackage{verbatim} % for multiline comment

\usepackage{paper_style}
\def\EE{\mathbb{E}}
\def\nm{\noalign{\medskip}}

\newcommand\dis\displaystyle
\def\NN{\mathbb{N}}

\title{Tracking of a Mobile Target Using Generalized Polarization Tensors\thanks{\footnotesize This work was supported by ERC Advanced Grant Project
MULTIMOD--267184 and NRF grants No. 2009-0090250 and
2010-0017532.}}
\author{}

\author{Habib Ammari\thanks{\footnotesize Department of Mathematics and Applications,
    Ecole Normale Sup\'erieure, 45 Rue d'Ulm, 75005 Paris, France
    (habib.ammari@ens.fr, boulier@dma.ens.fr, han.wang@ens.fr).} \and Thomas
  Boulier\footnotemark[2] \and Josselin Garnier\thanks{\footnotesize Laboratoire de
    Probabilit\'es et Mod\`eles Al\'eatoires \& Laboratoire Jacques-Louis Lions,
    Universit\'e Paris VII, 75205 Paris Cedex 13, France (garnier@math.jussieu.fr).}
  \and Hyeonbae Kang\thanks{Department of Mathematics, Inha University, Incheon 402-751,
    Korea (hbkang@inha.ac.kr).}  \and Han Wang\footnotemark[2]}

\begin{document}

\maketitle

\begin{abstract}
In this paper we apply an extended Kalman filter to track both the
location and the orientation of a mobile target  from multistatic
response measurements. We also analyze the effect of the
limited-view aspect on the stability and the efficiency of our
tracking approach. Our algorithm is based on the use of the
generalized polarization tensors, which can be reconstructed from
the multistatic response measurements by solving a linear system.
The system has the remarkable property that low order generalized
polarization tensors are not affected by the error caused by  the
instability of higher orders in the presence of measurement noise.
\end{abstract}

\bigskip

 \noindent {\footnotesize Mathematics Subject Classification
 (MSC2000): 35R30, 35B30}


\section{Introduction}
\label{sec:intro}

With each domain and material parameter, an infinite number of
tensors, called the Generalized Polarization Tensors (GPTs), is
associated. The concept of GPTs was introduced in
\cite{AK_SIMA_03, AK04}. The GPTs contain significant information
on the shape of the domain \cite{AGKLY11, AK_MMS_03,
ammari_generalized_2012-1}. It occurs in several interesting
contexts, in particular, in low-frequency scattering
\cite{dassios, AK04}, asymptotic models of dilute composites (see
\cite{milton} and \cite{AKT_AA_05}), in invisibility cloaking in
the quasi-static regime \cite{ammari_enhancement_2011-1} and in
potential theory related to certain questions arising in
hydrodynamics \cite{PS51}.

Another important use of this concept is for imaging diametrically
small conductivity inclusions from boundary or multistatic
response measurements. Multistatic response measurements are
obtained using arrays of point source transmitters and receivers.
This measurement configuration gives the so-called multistatic
response matrix (MSR), which measures the change in potential
field due to a conductivity inclusion. In fact, the GPTs are the
basic building blocks for the asymptotic expansions of the
perturbations of the MSR matrix due to the presence of small
conductivity inclusions inside a conductor \cite{FV_ARMA_89,CMV98,
AK_SIMA_03}. They can be reconstructed from the multi-static
response (MSR) matrix by solving a linear system. The system has
the remarkable property that low order generalized polarization
tensors are not affected by the error caused by  the instability
of higher orders in the presence of measurement noise. Based on
the asymptotic expansion, efficient and direct (non-iterative)
algorithms to determine the location and some geometric features
of the inclusions were proposed. We refer to \cite{AK04,
ammari_polarization_2007} and the references therein for recent
developments of this theory. An efficient numerical code for
computing the GPTs is described in \cite{yves}.

In \cite{ammari_target_2012}, we have analyzed the stability and
the resolving order of GPT in a circular full angle of view
setting with coincident sources and receivers, and developed
efficient algorithms for target identification from a dictionary
by matching the contracted GPTs (CGPTs). The CGPTs  are particular
linear combinations of the GPTs (called harmonic combinations) and
were first introduced in \cite{ammari_enhancement_2011-1}. As a
consequence, explicit relations between the CGPT of scaled,
rotated and translated objects have been established in
\cite{ammari_target_2012}, which suggest strongly that the GPTs
can also be used for tracking the location and the orientation of
a mobile object. One should have in mind that, in real
applications, one would like to localize the target and
reconstruct its orientation directly from the MSR data without
reconstructing the GPTs.

In this paper we apply an extended Kalman filter to track both the
location and the orientation of a mobile target directly from MSR
measurements.

The Extended Kalman Filter (EKF) is a generalization of the Kalman
Filter (KF) to nonlinear  dynamical systems. It is robust with
respect to noise and computationally inexpensive, therefore is
well suited for real-time applications such as tracking
\cite{kalman2}.

Target tracking is an important task in sonar and radar imaging,
security technologies,  autonomous vehicle, robotics, and
bio-robotics, see, for instance,
\cite{track1,track2,track3,track4,track5, track6}. An example in
bio-robotics is the weakly electric fish which has the faculty to
probe an exterior target with its electric dipole and multiple
sensors distributed on the skin \cite{ammari_modeling_2012}. The
fish usually swims around the target to acquire information. The
use of Kalman-type filtering for target tracking is quite
standard, see, for instance,
\cite{track1,track2,track3,track4,track5, track6}.

However, to the best of our knowledge, this is the first time
where tracking of the orientation of a target is provided.
Moreover, we analyze the ill-posed character of both the location
and orientation tracking in the case of limited-view data. In
practice, it is quite realistic to have the sources/receivers
cover only a limited angle of view. In this case, the
reconstruction of the GPTs becomes more ill-posed than in the
full-view case.

It is the aim of this paper to provide a fast algorithm for
tracking both the location and the orientation of a mobile target,
and precisely analyze the stability of the inverse problem in the
limited-view setting.

The paper is organized as follows. In section
\ref{sec:cond_prob_cgpt2} we recall the conductivity problem and
the linear system relating the CGPTs with the MSR data, and
provide a stability result in the full angle of view setting. In
section \ref{sec:tracking_of_mobile_target} we present a GPT-based
location and orientation tracking algorithm using an extended
Kalman filter and show the numerical results in the full-view
setting. In section \ref{sec:limited_angle_view} we analyze the
stability of the CGPT-reconstruction in the limited-view setting
and also test the performance of the tracking algorithm. The paper
ends with a few concluding remarks. An appendix is for a brief
review of the extended Kalman filter.

\section{Conductivity problem and reconstruction of CGPTs}
\label{sec:cond_prob_cgpt2}

We consider the two-dimensional conductivity problem.  Let $B$ be
a bounded $\mathcal{C}^2$-domain of characteristic size of order
$1$ and centered at the origin. Then $D=z+\delta B$ is an
inclusion of characteristic size of order $\delta$ and centered at
$z$. We denote by $0<\kappa \neq 1< +\infty$ its conductivity, and
$\lambda:=(\kappa+1)/(2\kappa - 2)$ its contrast. In the circular
setting, $N$ coincident sources/receivers are evenly spaced on the
circle of radius $R$ and centered at the origin $O$ between the
angular range $(0,\gamma]$. In the full-view case, $\gamma = 2\pi$
while $\gamma < 2\pi$ in the limited-view configuration. The
position of $s$-th source (and $r$-th receiver) is denoted by
$x_s$ (and $x_r$, respectively) for $s,r=1\ldots N$, with
$\theta_s=\gamma s/N$ the angular position. We require that the
circle is large enough to include the inclusion ($R>\delta$). In
the following, we set $\rho:=R/\delta>1$.

\subsection{CGPTs and the linear system}
\label{sec:CGPT_linsys}

In the presence of $D$, the electrical potential $u_s$ resulting
from a source at $x_s$ is given as the solution to the following
conductivity problem \cite{ammari_target_2012}:
\begin{align}
  \label{eq:simp_model}
  \left\{
    \begin{aligned}
      \nabla \cdot((1+(\kappa-1)\chi_D)\nabla u_s)(x) &=0, &x \in \R^2,\\
      u_s(x)-\Gamma(x-x_s) &= O(\abs{x}^{-1}), &\abs{x}\rightarrow
      +\infty,
    \end{aligned}
  \right.
\end{align}
where $\Gamma(x)= (1/2\pi)\, \log\abs{x}$ is the fundamental
solution of the Laplacian in $\R^2$: $\Delta \Gamma(x)=
\delta_0(x)$, with $\delta_0$ being the Dirac mass at $0$.

 Using asymptotic expansion of the
fundamental solution, the MSR data $\bV=(V_{sr})_{s,r}$ being
defined as $V_{sr}=u_s(x_r)-\Gamma(x_r -x_s)$, is linearly related
to the GPTs of $B$ as \cite{AK04,AK_SIMA_03}:
\begin{equation}
  V_{sr} = \sum_{|\alpha|,|\beta|=1}^{K} \frac{\delta^{\abs\alpha+\abs\beta}}{\alpha ! \beta !}
  \partial^\alpha \Gamma(z-x_s) M_{\alpha \beta}(\lambda, B)
  \partial^\beta \Gamma(z-x_r) + E_{sr} + W_{sr},
  \label{eq:MSR_GPT_relation}
\end{equation}
where $K$ denotes the highest order of GPTs in the expansion,
$\bE=(E_{sr})_{s,r}$ the truncation error (non-zero if
$K<\infty$), and $\bW=(W_{sr})_{s,r}$ the measurement noise
following independently the same normal distribution:
$W_{sr}\iid\Gaussian{0}{\stdnoise^2}$, of mean zero and variance
$\stdnoise^2$.

The contracted GPTs, being defined as a harmonic combination of
the GPTs \cite{ammari_enhancement_2011-1}, allow us to put
\eqref{eq:MSR_GPT_relation} into an equivalent form
\cite{ammari_target_2012}:
\begin{equation}
  \begin{aligned}
    V_{sr} &= \sum_{m,n=1}^K \underbrace{
      \frac{1}{2\pi m \rho^m} \begin{pmatrix} \displaystyle
        {\cos m\theta_s}, \displaystyle {\sin
          m\theta_s} \end{pmatrix}}_{\mathbf{A}_{sm}}
    \underbrace{\begin{pmatrix}
        \mathbf{M}^{cc}_{mn} & \mathbf{M}^{cs}_{mn} \\
        \mathbf{M}^{sc}_{mn} & \mathbf{M}^{ss}_{mn}
      \end{pmatrix}}_{\mathbf{M}_{mn}}
    \underbrace{\begin{pmatrix}
        \cos n\theta_r \\
        \sin n\theta_r
      \end{pmatrix}
      \frac{1}{2\pi n \rho^n}}_{(\mathbf{A}_{rn})^\top} + E_{sr} +
      W_{sr},
  \end{aligned}
  \label{eq:MSR_CGPT_linsys_coeffwise}
\end{equation}
where $^\top$ denotes the transpose and the CGPT matrix
$\bM=(\bM_{mn})_{m,n}$ \footnote{Throughout the paper, we will
write
  $\bM_{mn}$ for the $m,n$-th $2\times 2$ building block, and $(\bM)_{ab}$ for the $a,b$-th entry in
  $\bM$.} has dimension $2K\times 2K$.

Recall that $\bA=\mbf C\mbf D$, with $\bC$ being a $N \times 2K$ matrix constructed
from the block $\bC_{rm} = (\cos m \theta_r , \sin m\theta_r)$ and $\bD$ a $2K \times
2K$ diagonal matrix:
\begin{equation} \label{defDC}
  \bC = \begin{pmatrix}
    \bC_{11} & \bC_{12} & \cdots & \bC_{1K}\\
    \bC_{21} & \bC_{22} & \cdots & \bC_{2K}\\
    \cdots & \cdots & \ddots & \cdots\\
    \bC_{N1} & \bC_{N2} & \cdots & \bC_{NK}
  \end{pmatrix};
  \bD = \frac{1}{2\pi} \begin{pmatrix}
    \mathbf{I}_2/\rho&  &  &  \\
    &   \mathbf{I}_2/(2 \rho^2) &   &  \\
    &   & \ddots &  \\
    &   &   & \mathbf{I}_2/(K \rho^K)
  \end{pmatrix}.
\end{equation}
Here, $\bI_2$ is the $2\times 2$ identity matrix. With these
notations in hand, we introduce the linear operator
%$\bL(\bM)=\bC\bD \bM \bD \bC^\top$
\begin{align}
  \label{eq:op_L}
  \bL(\bM)=\bC\bD \bM \bD \bC^\top,
\end{align}
and rewrite \eqref{eq:MSR_CGPT_linsys_coeffwise} as:
\begin{align}
  \label{eq:MSR_CGPT_linsys}
  %V=\bL(\bM)+\mbf E = \bA \bM \bA^\top + \mbf E
  \bV = \bL(\bM) + \bE + \bW .
\end{align}
In order to reconstruct $\bM$, we solve the least-squares problem:
\begin{align}
  \label{eq:Least_square_CGPT_recon}
  \min_{\bM}\  \norm{\bL(\bM) - \bV}_F^2,
\end{align}
where $\norm{\cdot}_F$ denotes the Frobenius norm. It is well
known that \eqref{eq:Least_square_CGPT_recon} admits a unique
minimal norm solution $\bMest=\bL^\dagger(\bV)$, with
$\bL^\dagger$ being the pseudo-inverse of $\bL$ provided by the
following lemma:
\begin{lem}\label{lem:pseudo-inverse-AxB}
  Let $\bA, \bB$ be two real matrices of arbitrary dimension, and define the linear operator
  $\bL(\bX)=\bA \bX \bB^\top$. If $\bA^\dagger, \bB^\dagger$ are the pseudo-inverse of $\bA, \bB$ respectively, then the
  pseudo-inverse of $\bL$ is given by
  \begin{equation}
    \label{eq:pseudo_inverse}
    \bL^\dagger(\bY) = \bA^\dagger \bY (\bB^\dagger)^\top .
  \end{equation}
\end{lem}
\begin{proof}
  This is a straightforward verification of the definition of pseudo-inverse, namely: 1) $\bL^\dagger
  \bL$ and $\bL \bL^\dagger$ are self-adjoint; 2) $\bL \bL^\dagger \bL = \bL$ and $\bL^\dagger \bL \bL^\dagger =
  \bL^\dagger$. For the first point:
  \begin{align*}
    \pinv \bL(\bL(\bX)) = \pinv \bA \bA \bX(\pinv \bB \bB)^\top,
  \end{align*}
  which is self-adjoint since the matrices $\pinv \bA \bA$ and $\pinv \bB \bB$ are symmetric by definition
  of pseudo-inverse; while for the second point, it follows from the
  definition again that
  \begin{align*}
    \bL(\pinv \bL(\bL(\bX))) = \bA\pinv \bA \bA \bX(\bB\pinv \bB \bB)^\top = \bA \bX\T \bB =
    \bL(\bX).
  \end{align*}
  Similarly, one can verify the self-adjointness of $\bL\pinv \bL$ and $\bL^\dagger \bL \bL^\dagger = \bL^\dagger$.
\end{proof}

\subsection{Full-view setting}
\label{sec:full_angle_view}

In  \cite{ammari_target_2012}, we have investigated the resolving
order of CGPT reconstruction in the full angle of view setting:
$\gamma=2\pi$. Given $N\geq 2K$, it has been shown that the matrix
$\bC$ is orthogonal (up to the factor $N/2$): $$\bC^\top \bC=\frac
N 2\mbf {I},$$ and the pseudo-inverse solution takes the form:
\begin{align}
  \label{eq:pseudo_L_full_aov}
  \bL^\dagger(\bV)=\frac{4}{N^2}\inv\bD\bC^\top \bV \bC \inv\bD .
\end{align}
Furthermore, the reconstruction problem is exponentially
ill-posed. More precisely, the following result holds.
\begin{prop}\label{prop:full-angle-view-svd}
Let $\mathbf{e}_{ab}$ be the $2K\times2K$ matrix whose elements
are all zero but the $(a,b)$th element is equal to $1$.
  In the circular and full-view setting with $N\geq2K$, the $(a,b)$-th singular value of the
  operator $\bL$, for $a,b=1, \ldots, 2K$,  is
\begin{equation} \label{singv}
  \lambda_{ab}=N/(8\pi^2 \ceil{a/2}\ceil{b/2}\rho^{\ceil{a/2}+\ceil{b/2}}),
  \end{equation} with
  the matrix $\mathbf{e}_{ab}$ as the right singular vector, and
$\mathbf{f}_{ab}=\lambda_{ab}^{-1}\bL(\mathbf{e}_{ab})$
  as the left singular vector. In particular, the condition number of the operator $\bL$
  is
  $K^2\rho^{2(K-1)}$. %for $a,b=1\ldots N$.
\end{prop}
\begin{proof}
  Using the fact that $\bC^\top \bC=\frac N 2\mbf {I}$, we have, for any square matrices $\bU$ and
  $\bV$,
  \begin{align}
    \label{eq:L_innerprod}
    \Seq{\bL(\bU),\bL(\bV)} = \frac{N^2}{4}\Seq{\bD \bU \bD, \bD \bV \bD},
  \end{align}
  where $\seq{\cdot,\cdot}$ is the termwise inner product. Since $\bD$ is diagonal and invertible,
  we conclude that the canonical basis $\set{\mathbf{e}_{ab}}_{a,b}$ is the singular vector of $\bL$, and the
  associated singular value is $\norm{\bL(\mathbf{e}_{ab})}_F = \norm{\bD \mathbf{e}_{ab} \bD}_F N/2 = N/(8\pi^2
  \ceil{a/2}\ceil{b/2}\rho^{\ceil{a/2}+\ceil{b/2}})$.
\end{proof}
As a simple consequence, we have $\pinv\bL(\bW)_{ab} =
\lambda_{ab}^{-1}\seq{\bW,f_{ab}}$. When $K$ is sufficiently
large, the truncation error $\bE$ is $O(\rho^{-K-2})$ and can be
neglected if compared to $\bW$ \cite{ammari_target_2012}, and then
by the property of white noise
\begin{align*}
  \sqrt{\Exp{\Paren{(\bMest)_{ab} - (\bM)_{ab}}^2}} \lesssim
  \sqrt{\Exp{\Paren{\pinv\bL(\bW)_{ab}^2}}}
  = \lambda_{ab}^{-1}\stdnoise ,
\end{align*}
which is the result already established in
\cite{ammari_target_2012}. Hence, it follows from (\ref{singv})
that the reconstruction of high order CGPTs is an ill-posed
problem. Nonetheless the system has the remarkable property that
low order CGPTs are not affected by the error caused by the
instability of higher orders as the following proposition shows.
\begin{prop}
  Let $\bM_K$ denote the CGPTs of order up to $K$, and let $\bL_K$ be the corresponding linear operator in
  \eqref{eq:MSR_CGPT_linsys_coeffwise}. Then, for any order $K_1\leq K_2<N/2$, the submatrix of
  $\bL_{K_2}^\dagger(\bV)$ formed by the first $2K_1$ columns and rows is identical to the minimal
  norm solution $\bL_{K_1}^\dagger(\bV)$.
\end{prop}
\begin{proof}
  Let the $N\times 2K$ matrix $J_K$ be the row concatenation of the $2K\times 2K$ identity matrix $\bI_{2K}$ and
  a zero matrix.
We have $\bJ_K^\top \bJ_K=\bI_{2K}$ and $\bJ_{K_1}^\top
\bL_{K_2}^\dagger(\bV) \bJ_{K_1}$ is the submatrix of
$\bL_{K_2}^\dagger(\bV)$ formed by the first $2K_1$ columns and
rows. Let $\bD_K$  and $\bC_K$ be the matrices defined in
(\ref{defDC}). Because of (\ref{eq:pseudo_L_full_aov}), we have
$$
\bJ_{K_1}^\top \bL_{K_2}^\dagger(\bV) \bJ_{K_1} =  \frac{4}{N^2}
\bJ_{K_1}^\top \bD_{K_2}^{-1} \bC_{K_2}^\top \bV  \bC_{K_2}
\bD_{K_2}^{-1}  \bJ_{K_1}.
$$
One can easily see that
$$
\bC_{K_2} \bD_{K_2}^{-1}  \bJ_{K_1} = \bC_{K_1} \bD_{K_1}^{-1}.
$$
Thus, we have
$$
\bJ_{K_1}^\top \bL_{K_2}^\dagger(\bV) \bJ_{K_1} =
\bL_{K_1}^\dagger(\bV).
$$
\end{proof}
Numerically,  $\pinv\bL$ can be implemented through either the
formula \eqref{eq:pseudo_L_full_aov} or the Conjugated Gradient
(CG) method using  (\ref{eq:Least_square_CGPT_recon}). Simulations
in \cite{ammari_target_2012} confirm that in typical situations,
say, with $K=5$ and $10\%$ noise, the reconstructed CGPT is
sufficiently accurate for the task such as the target
identification in a dictionary. In the next section we present a
location and orientation tracking algorithm for a mobile target
based on the concept of CGPTs.

\section{Tracking of a mobile target}
\label{sec:tracking_of_mobile_target}

% \paragraph{Time dependent data acquisition}
At the instant $t\geq 0$, we denote by $z_t=[x_t, y_t]^\top\in\R^2$ the location and
$\theta_t\in[0,2\pi)$ the orientation of a target $D_t$.
\begin{align}
  \label{eq:target_transrot_t}
  D_t=z_t+R_{\theta_t} D,
\end{align}
where $R_{\theta_t}$ is the rotation by $\theta_t$. Let $\bM_t$ be the CGPT of $D_t$,
and $\bMD$ be the CGPT of $D$. Then the equation \eqref{eq:MSR_CGPT_linsys} becomes:
\begin{align}
  \label{eq:MSR_CGPT_linsys_time}
  \bV_t = \bL (\bM_t) + \mbf E_t + \bW_t,
\end{align}
where $\bE_t$ is the truncation error, and $\bW_t$ the measurement noise at time $t$.

The objective of \emph{tracking} is to estimate the target's location $z_t$ and
orientation $\theta_t$ from the MSR data stream $\bV_t$. We emphasize that these
informations are contained in the first two orders CGPTs as shown in the previous
paper \cite{ammari_target_2012}. Precisely, let $\Delta x_t=x_t-x_{t-1}$, $\Delta
y_t=y_t-y_{t-1}$ and $\Delta \theta_t=\theta_t-\theta_{t-1}$, then the following
relations (when it is well defined) exist between the CGPT of $D_t$ and $D_{t-1}$
\cite{ammari_target_2012}:
\begin{equation}
  \label{eq:linsys_P1}
  \begin{aligned}
    \No_{12}(D_t)/\No_{11}(D_t) &= 2(\Delta x_t + i\Delta y_t) + e^{i\Delta\theta_t} \No_{12}(D_{t-1})/\No_{11}(D_{t-1}),\\
    \nm
    \Nt_{12}(D_t)/\Nt_{11}(D_t) &= 2(\Delta x_t + i\Delta y_t) + e^{i\Delta\theta_t} \Nt_{12}(D_{t-1})/\Nt_{11}(D_{t-1}).
  \end{aligned}
\end{equation}
Hence when the linear system \eqref{eq:linsys_P1} is solvable, one
can estimate $z_t, \theta_t$ by solving and accumulating $\Delta
x_t, \Delta y_t$ and $\Delta \theta_t$. However, such an algorithm
will propagate the error over time, since the noise presented in
data is not properly taken into account here.

In the following we develop a CGPT-based tracking algorithm using the Extended Kalman
Filter, which handles correctly the noise. We recall first the definition of complex
CGPT, with which a simple relation between $\bM_t$ and $\bMD$ can be established.

\subsection{Time relationship between CGPTs}
Let $\oov=(1, i)^\top$. The complex CGPTs $\No, \Nt$ are defined
by
\begin{align*}
  \No_{mn} &= (\Mcc_{mn} - \Mss_{mn}) + i (\Mcs_{mn} + \Msc_{mn}) = \oov^\top \bM_{mn} \oov, \\
  \Nt_{mn} &= (\Mcc_{mn} + \Mss_{mn}) + i (\Mcs_{mn} - \Msc_{mn}) = \oov^H \bM_{mn}
  \oov,
\end{align*}
where $H$ denotes the Hermitian transpose. Therefore, we have
\begin{align}
  \label{eq:CCGPT_CGPT_mat}
  \No = \bU^\top \bM \bU  \quad \mbox{ and } \quad \Nt = \bU^H \bM \bU,
\end{align}
where the matrix $\bU$ of dimension $2K\times K$ over the complex
fields is defined by
\begin{align}
  \label{eq:def_U_mat}
  \bU =\begin{pmatrix}
    \oov & 0 & \ldots & 0\\
    0 & \oov & \ldots & 0\\
    \vdots & ~ & \ddots & \vdots \\
    0 & \ldots & 0 & \oov
  \end{pmatrix} .
\end{align}
It is worth mentioning that $\No$ and $\Nt$ are complex matrices
of dimension $K\times K$.

To recover the CGPT $\bM_{mn}$ from the complex CGPTs $\No, \Nt$,
we simply use the relations
\begin{equation}
  \label{eq:CGPT_CCGPT_recovery}
  \begin{aligned}
    \Mcc_{mn} &= \frac{1}{2} \Re(\No_{mn} + \Nt_{mn}), \ \Mcs_{mn} = \frac{1}{2} \Im(\No_{mn} + \Nt_{mn}),\\
    \Msc_{mn} &= \frac{1}{2} \Im(\No_{mn} - \Nt_{mn}), \ \Mss_{mn} = \frac{1}{2} \Re(\Nt_{mn} -
    \No_{mn}),
  \end{aligned}
\end{equation}
where $\Re,\Im$ are the real and imaginary part of a complex
number, respectively. For two targets $D_t, D$ satisfying
\eqref{eq:target_transrot_t}, the following relationships between
their complex CGPT hold \cite{ammari_target_2012}:
\begin{subequations}
  \label{eq:CCGPT_tsr}
  \begin{align}
    \No(D_t) &= \Fmat^\top \No(D) \Fmat, \\
    \Nt(D_t) &= \Fmat^H \Nt(D) \Fmat ,
    % \No(D_t) &= (\Ft) \No(D_0) (\Ft)^\top, \\
    % \Nt(D_t) &= (\overline\Ft) \Nt(D_0) (\Ft)
  \end{align}
\end{subequations}
where $\Fmat$ is a upper triangle matrix with the $(m,n)$-th entry given by
\begin{align}
  \label{eq:Fmat_mn}
  (\Fmat)_{mn} = \binom{n}{m} {(x_t + i y_t)}^{n-m}e^{im\theta_t} .
\end{align}

\paragraph{Linear operator $\mbf T_t$:}
Now one can find explicitly a linear operator $\mbf T_t$ (the
underlying scalar field is $\R$) which depends only on $z_t,
\theta_t$, such that $\bM_t=\mbf T_t(\bMD)$, and the equation
\eqref{eq:MSR_CGPT_linsys_time} becomes
\begin{align}
  \label{eq:MSR_CGPT_linsys_time_Tt}
  \bV_t = \bL(\bT_t(\bMD)) + \bE_t + \bW_t .
\end{align}
For doing so, we set $\Jmat := \bU\Fmat$, where $\bU$ is given by
(\ref{eq:def_U_mat}). Then, a straightforward computation using
\eqref{eq:CCGPT_CGPT_mat}, \eqref{eq:CGPT_CCGPT_recovery}, and
\eqref{eq:CCGPT_tsr}  shows that
\begin{equation}
  \label{eq:CGPT_tsr}
  \begin{aligned}
    \Mcc(D_t) &= \Re \Jmat^\top \bMD \Re \Jmat, \ \Mcs(D_t) = \Re \Jmat^\top \bMD \Im \Jmat, \\
    \Msc(D_t) &= \Im \Jmat^\top \bMD \Re \Jmat, \ \Mss(D_t) = \Im \Jmat^\top \bMD \Im
    \Jmat,
  \end{aligned}
\end{equation}
where $\Mcc(D_t),\Mcs(D_t),\Msc(D_t),\Mss(D_t)$ are defined in
(\ref{eq:MSR_CGPT_linsys_coeffwise}). Therefore, we get the
operator $\mbf T_t$:
\begin{align}
  \label{eq:def_op_Tt}
  \mbf T_t(\bMD) =\ &\Re \bU (\Re \Jmat^\top \bMD \Re \Jmat) \Re \bU^\top + \Re \bU (\Re \Jmat^\top
  \bMD \Im \Jmat) \Im \bU^\top + \notag \\
  &\Im \bU (\Im \Jmat^\top \bMD \Re \Jmat) \Re \bU^\top + \Im \bU (\Im \Jmat^\top \bMD \Im \Jmat) \Im
  \bU^\top = \bM_t .
\end{align}

\subsection{Tracking by the Extended Kalman Filter}
\label{sec:track-extend-kalm} The EKF is a generalization of the
KF to nonlinear dynamical systems. Unlike KF which is an optimal
estimator for linear systems with Gaussian noise, EKF is no longer
optimal, but it remains robust with respect to noise and
computationally inexpensive, therefore is well suited for
real-time applications such as tracking. We establish here the
\emph{system state} and the \emph{observation} equations which are
fundamental to EKF, and refer readers to Appendix
\ref{sec:extend-kalm-filt} for its algorithmic details.

\subsubsection{System state observation equations}
\label{sec:system_state_obsrv_eq}

We assume that the position of the target is subjected to an external driving force
that has the form of a white noise. In other words the velocity $({V}(\tau))_{\tau
  \in \R^+}$ of the target is given in terms of a two-dimensional Brownian motion
$({W}_a(\tau))_{\tau \in \R^+}$ and its position $({Z}(\tau))_{\tau \in \R^+}$ is
given in terms of the integral of this Brownian motion:
$$
{V}(\tau) = {V}_0 + \sigma_a  { W}_a(\tau), \quad \quad {Z}(\tau) = {Z}_0 + \int_0^\tau {V}(s) ds .
$$
The orientation  $(\Theta(\tau))_{\tau \in \R^+}$ of the target is subjected to random fluctuations and
its angular velocity is
given in terms of an independent white noise, so that the orientation
is  given in terms  of a one-dimensional Brownian motion $(W_\theta(\tau))_{\tau \in \R^+}$:
$$
\Theta(\tau) = \Theta_0+ \sigma_\theta W_\theta(\tau) .
$$
We observe the target at discrete times $t \Delta \tau$, $t \in \NN$, with time step $\Delta \tau$.
We denote ${z}_t=  {Z}(t \Delta \tau)$, ${v}_t =  {V}(t \Delta \tau)$, and $\theta_t=  \Theta(t \Delta \tau)$.
They obey the recursive relations
\begin{equation} \label{eq:move_model}
  \begin{array}{ll}
\dis  {v}_t = {v}_{t-1}+ {a}_{t} ,
&\quad \dis {a}_{t}= \sigma_a  \big( {W}_a(t\Delta \tau) -{W}_a((t-1)\Delta \tau) \big), \\
\dis  {z}_t = {z}_{t-1}+ {v}_{t-1} \Delta \tau + {b}_{t} , &\quad
\dis {b}_{t}= \sigma_a  \int_{(t-1) \Delta \tau}^{t \Delta \tau}
{W}_a(s) -{W}_a((t-1)\Delta \tau)ds , \\
\dis  \theta_t = \theta_{t-1}+ c_{t} , &\quad \dis c_{t}=
\sigma_\theta  \big( W_\theta(t\Delta \tau) -W_\theta((t-1)\Delta
\tau) \big) .
\end{array}
\end{equation}
 Since the increments of the Brownian motions are independent from each other,
 the vectors $(U_t)_{t \geq 1}$ given by
 $$
 U_t=
 \begin{pmatrix}
 {a}_t\\
 {b}_t\\
 c_t\end{pmatrix}
 $$
 are independent and identically distributed
 with the multivariate normal distribution with mean zero and
 covariance matrix $\bSigma$ given by
 \begin{equation} \label{eq:noise_cov_matrix}
 \bSigma= \Delta \tau
 \begin{pmatrix}
 \sigma_a^2 {\bf I}_2 & \frac{\sigma_a^2}{2} \Delta \tau {\bf I}_2 &0\\
 \frac{\sigma_a^2}{2} \Delta \tau {\bf I}_2 & \frac{\sigma_a^2}{3} \Delta \tau^2 {\bf I}_2 & 0\\
 0 & 0 & \sigma_\theta^2
 \end{pmatrix}
 \end{equation}
 The evolution of the state vector
 $$
 X_t =
 \begin{pmatrix}
 {v}_t\\
 {z}_t\\
 \theta_t
 \end{pmatrix}
 $$
 takes the form
 \begin{equation} \label{eq:sys_state_matrix_form}
 X_t = \bF X_{t-1} + U_{t}, \quad \quad
 \bF=
 \begin{pmatrix}
 \bI_2 & 0 & 0\\
 \Delta \tau \bI_2 & \bI_2 & 0 \\
 0 & 0 & 1
 \end{pmatrix}
 \end{equation}

The observation made at time $t$ is the MSR matrix given by
\eqref{eq:MSR_CGPT_linsys_time_Tt}, where the system state $X_t$
is implicitly included in the operator $\bT_t$. We suppose that
the truncation error $\bE_t$ is small compared to the measurement
noise so that it can be dropped in
\eqref{eq:MSR_CGPT_linsys_time_Tt}, and that the Gaussian white
noise $\bW_t$ of different time are mutually independent. We
emphasize that the velocity vector $v_t$ of the target does not
contribute to \eqref{eq:MSR_CGPT_linsys_time_Tt}, which can be
seen from \eqref{eq:target_transrot_t}. To highlight the
dependence upon $z_t, \theta_t$, we introduce a function $h$ which
is nonlinear in $z_t, \theta_t$, and takes $\bMD$ as a parameter,
such that
\begin{align}
  \label{eq:linsys_func_h}
  % h(z_t, \theta_t;\bMD) = \bL(\bT_t(\bMD)),
  h(X_t; \bMD) = h(z_t, \theta_t;\bMD) = \bL(\bT_t(\bMD)).
\end{align}
Then together with \eqref{eq:sys_state_matrix_form} we get the following \emph{system state} and
\emph{observation} equations:
\begin{subequations}
  \label{eq:sys_state_obs_eq}
  \begin{align}
      X_t &= \bF X_{t-1} + U_t , \label{eq:system_eq} \\
      \bV_t &= h(X_t; \bMD) + \bW_t .\label{eq:observ_eq}
      % \bV_t &= \bL_t(\bT_t(\bMD)) + \bW_t \label{eq:observ_eq}
  \end{align}
\end{subequations}
Note that \eqref{eq:system_eq} is linear, so in order to apply EKF on \eqref{eq:sys_state_obs_eq},
we only need to linearize \eqref{eq:observ_eq}, or in other words, to calculate the partial
derivatives of $h$ with respect to $x_t, y_t, \theta_t$.

\subsubsection{Linearization of the observation equation}
\label{sec:lin_obs_eq}

Clearly, the operator $\bL$ contains only the information
concerning the acquisition system and does not depend on $x_t,
y_t, \theta_t$. So by \eqref{eq:linsys_func_h}, we have
\begin{align}
  \label{eq:partial_h}
  \partial_{x_t}h=\bL(\partial_{x_t}\bT_t(\bMD)),
%, \ \ \partial_{\theta_t}h=\bL_t(\partial_{\theta_t}\bT_t(\bMD)),
\end{align}
while the calculation for $\partial_{x_t}\bT_t$ is straightforward using
\eqref{eq:def_op_Tt}. We have %, see Appendix C for details.
\begin{align}
  \label{eq:derivatives_Tt}
  \partial_{x_t}\bT_t(\bMD) = &\Re \bU \partial_{x_t}(\Re \Jmat^\top \bMD \Re \Jmat) \Re \bU^\top +
  \Re \bU \partial_{x_t}(\Re \Jmat^\top \bMD \Im \Jmat) \Im \bU^\top + \notag \\
  &\Im \bU \partial_{x_t}(\Im \Jmat^\top \bMD \Re \Jmat) \Re \bU^\top + \Im \bU \partial_{x_t}(\Im
  \Jmat^\top \bMD \Im \Jmat) \Im \bU^\top ,
\end{align}
where the derivatives are found by the chain rule:
\begin{align*}
\partial_{x_t}(\Re \Jmat^\top \bMD \Re \Jmat) &= \Re (\partial_{x_t}\Jmat^\top) \bMD \Re
\Jmat+\Re\Jmat^\top \bMD \Re(\partial_{x_t}\Jmat),\\
\partial_{x_t}(\Re \Jmat^\top \bMD \Im \Jmat) &= \Re (\partial_{x_t}\Jmat^\top) \bMD
\Im \Jmat+\Re\Jmat^\top \bMD \Im(\partial_{x_t}\Jmat),\\
\partial_{x_t}(\Im \Jmat^\top \bMD \Re \Jmat) &= \Im (\partial_{x_t}\Jmat^\top)
\bMD \Re \Jmat+\Im\Jmat^\top \bMD \Re(\partial_{x_t}\Jmat),\\
\partial_{x_t}(\Im \Jmat^\top \bMD \Im \Jmat) &= \Im (\partial_{x_t}\Jmat^\top)
 \bMD \Im \Jmat+\Im\Jmat^\top \bMD \Im(\partial_{x_t}\Jmat),
\end{align*}
and $\partial_{x_t}\Jmat = \bU\partial_{x_t} \mathbf{F}_t$. The
$(m,n)$-th entry of the matrix $\partial_{x_t} \mathbf{F}_t$ is
given by
\begin{align}
  \label{derv_z_Ft}
  (\partial_{x_t} \mathbf{F}_t)_{m,n} = \binom{n}{m} (n-m) z_t^{n-m-1}
  e^{im\theta_t}.
\end{align}
The derivatives $ \partial_{y_t}\bT_t(\bMD)$ and $\partial_{\theta_t}\bT_t(\bMD)$ are calculated
in the same way.

\subsection{Numerical experiments of tracking in the full-view setting}
\label{sec:tracking_numexp_full_view}

Here we show the performance of EKF in a full angle of view
setting with the shape 'A' as target $D$, which has diameter 10
and is centered at the origin. The path $(z_t, \theta_t)$ is
simulated according to the model \eqref{eq:move_model} during a
period of 10 seconds ($\Delta \tau=0.01$), with parameters
$\sigma_a=2, \sigma_\theta=0.5$, and the initial state $X_0=
(v_0,z_0,\theta_0)^\top =(-1, 1, 5, -5, 3\pi/2)^\top$. We make
sure that the target is always included inside the measurement
circle on which $N=20$ sources/receivers are fixed, see
Fig.~\ref{fig:target_path}. The data stream $\bV_t$ is generated
by first calculating the MSR matrix corresponding to each $D_t,
t\geq 0$ then adding a white noise.

Suppose that the CGPT of $D$ is correctly determined (for
instance, by identifying the target in a dictionary
\cite{ammari_target_2012}). Then we use the first two orders CGPT
$\bMD$ of $D$ in \eqref{eq:observ_eq}, and take $(0, 0, 10, -0.5,
0)^\top$ as initial guess of $X_0$ for EKF.

We add $10\%$ and $20\%$ of noise to data, and show the results of tracking in
%The results of tracking for the noise level $\stdnoise=0.1$ and $\stdnoise=0.2$ are shown in
Fig.~\ref{fig:tracking_big_small_target} (a) (c) and (e). We see that EKF can find
the true system state, despite of the poor initial guess, and the tracking precision
decays as the measurement noise level gets higher. The same experiment with small
target (of same shape) of diameter 1 is repeated in
Fig.~\ref{fig:tracking_big_small_target} (b) (d) and (f), where the tracking of
position remains correct, on the contrary, that of orientation fails when the noise
level is high. Such a result is in accordance with physical intuitions. In fact, the
position of a small target can be easily localized in the far field, while its
orientation can be correctly determined only in the near field.

%\graphicspath{{../figures/tracking_full_aov/}}

\begin{figure}[htp]
  \centering
  \includegraphics[width=7.5cm]{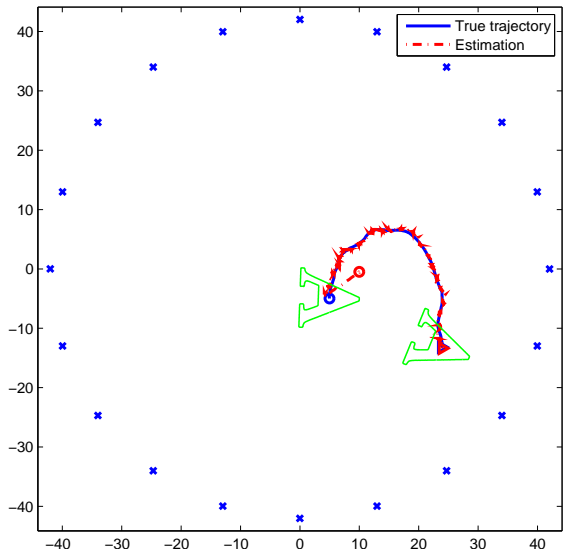}
  %\vspace{7.5cm}
  \caption{Trajectory of the letter 'A' and the estimation by EKF. The initial position is $(5, -5)$
    while the initial guess given to EKF is $(10, -0.5)$. The crosses indicate the position of
    sources/receivers, while the circle and the triangle indicate the starting and the final
    position of the target, respectively. In blue is the true trajectory and in red the estimated one.}
  \label{fig:target_path}
\end{figure}

%\graphicspath{{tracking_full_aov/}}

\begin{figure}[htp]
  \centering
  \subfigure[]{\includegraphics[width=6.5cm]{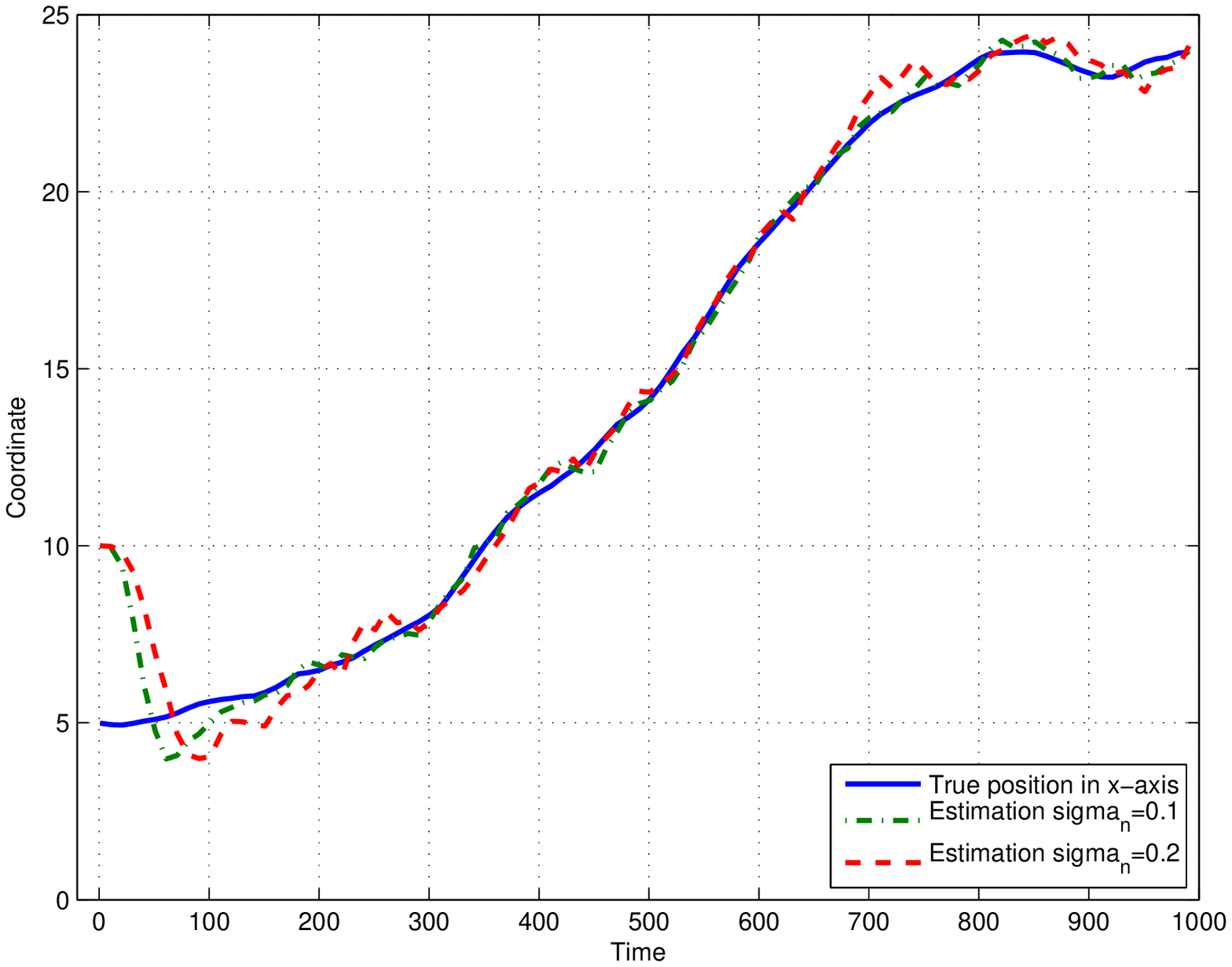}}
  \subfigure[]{\includegraphics[width=6.5cm]{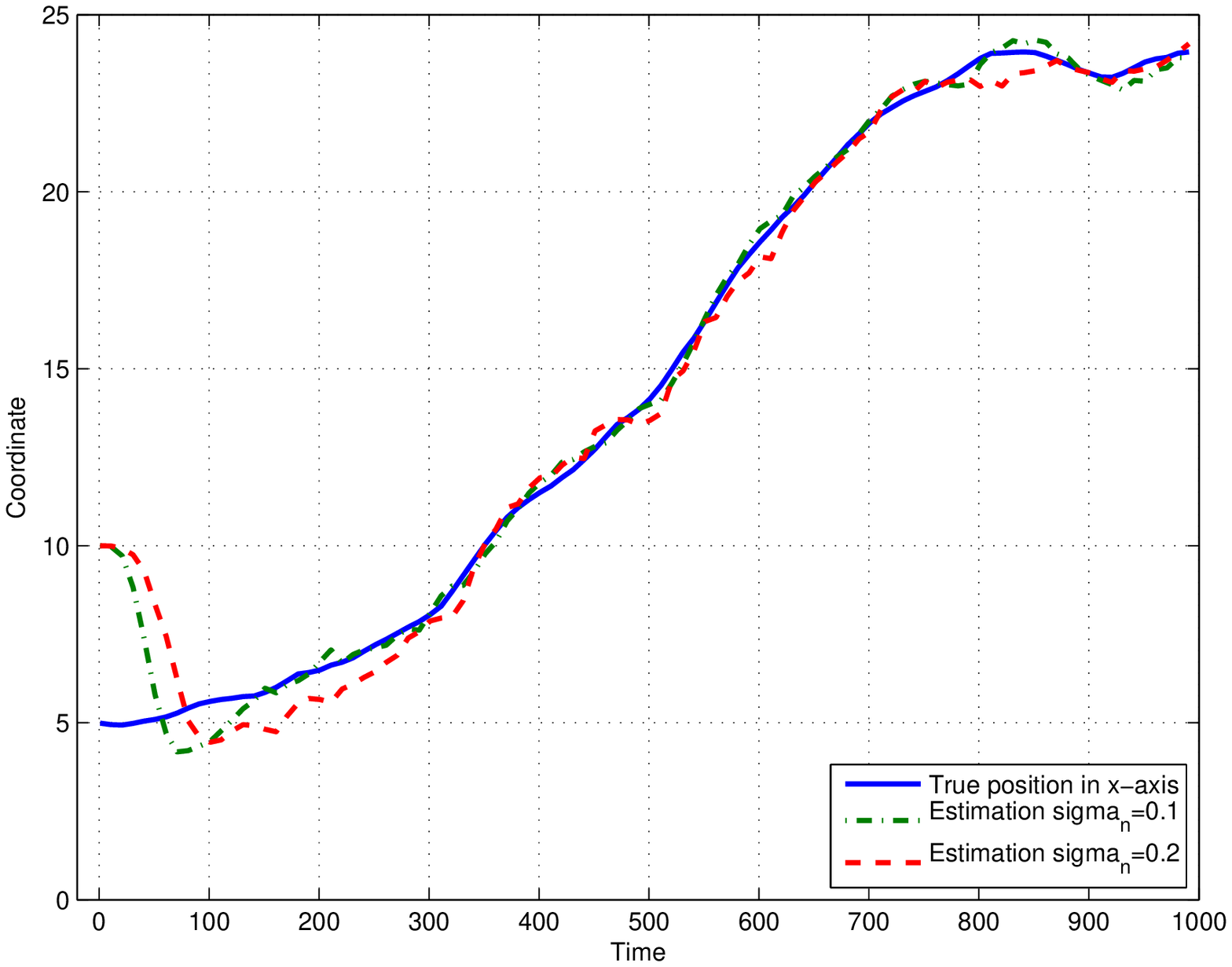}}
  \subfigure[]{\includegraphics[width=6.5cm]{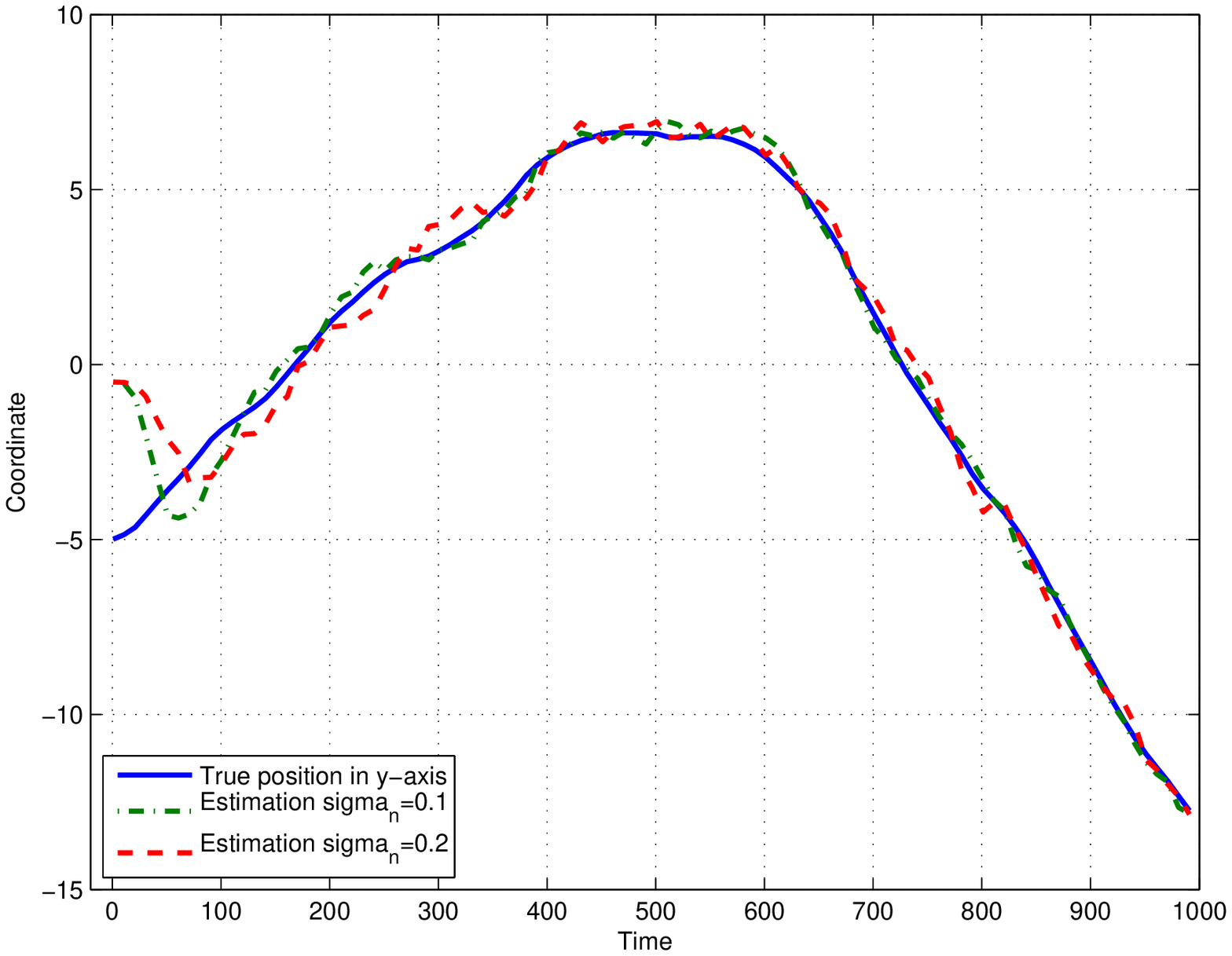}}
  \subfigure[]{\includegraphics[width=6.5cm]{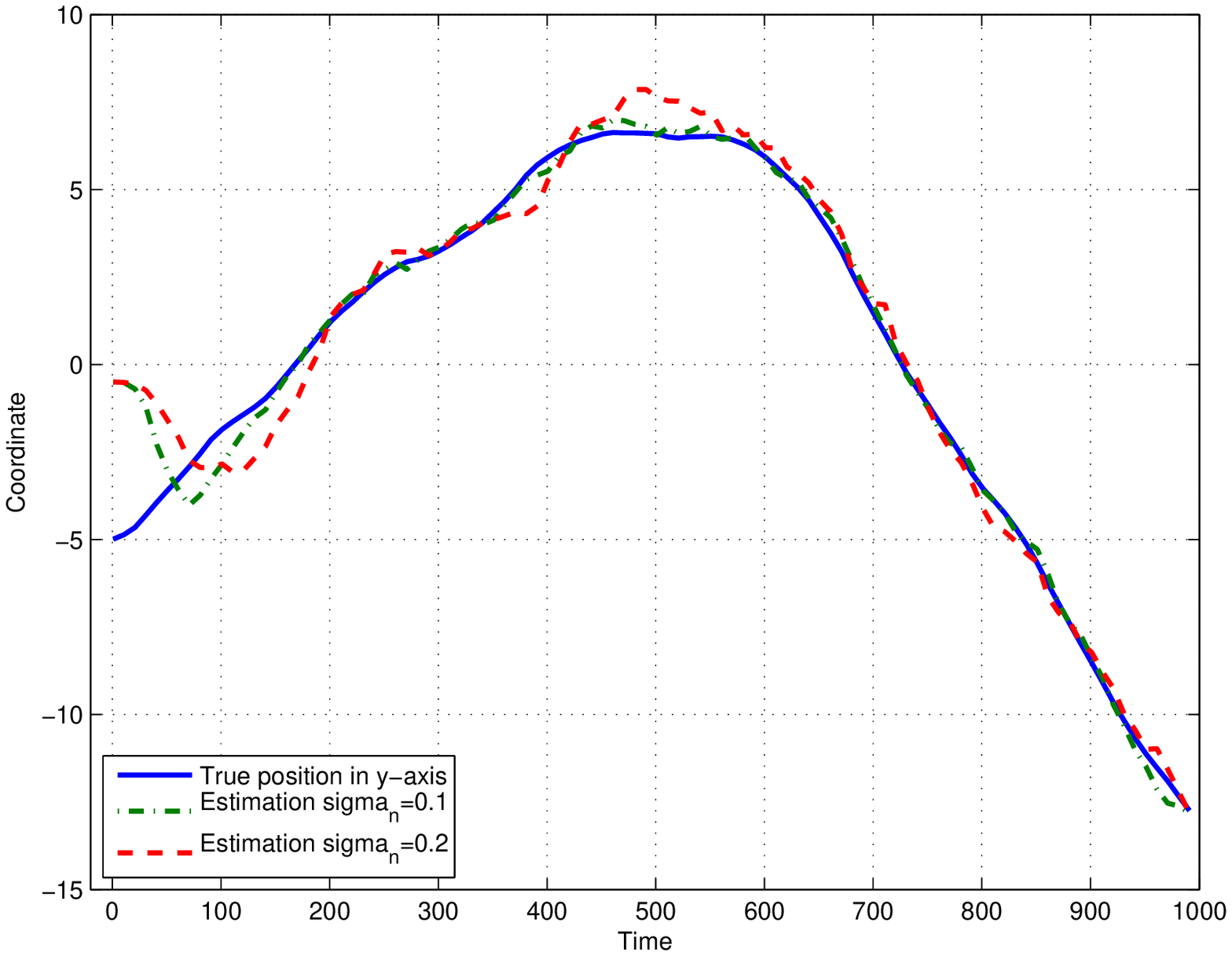}}
  \subfigure[]{\includegraphics[width=6.5cm]{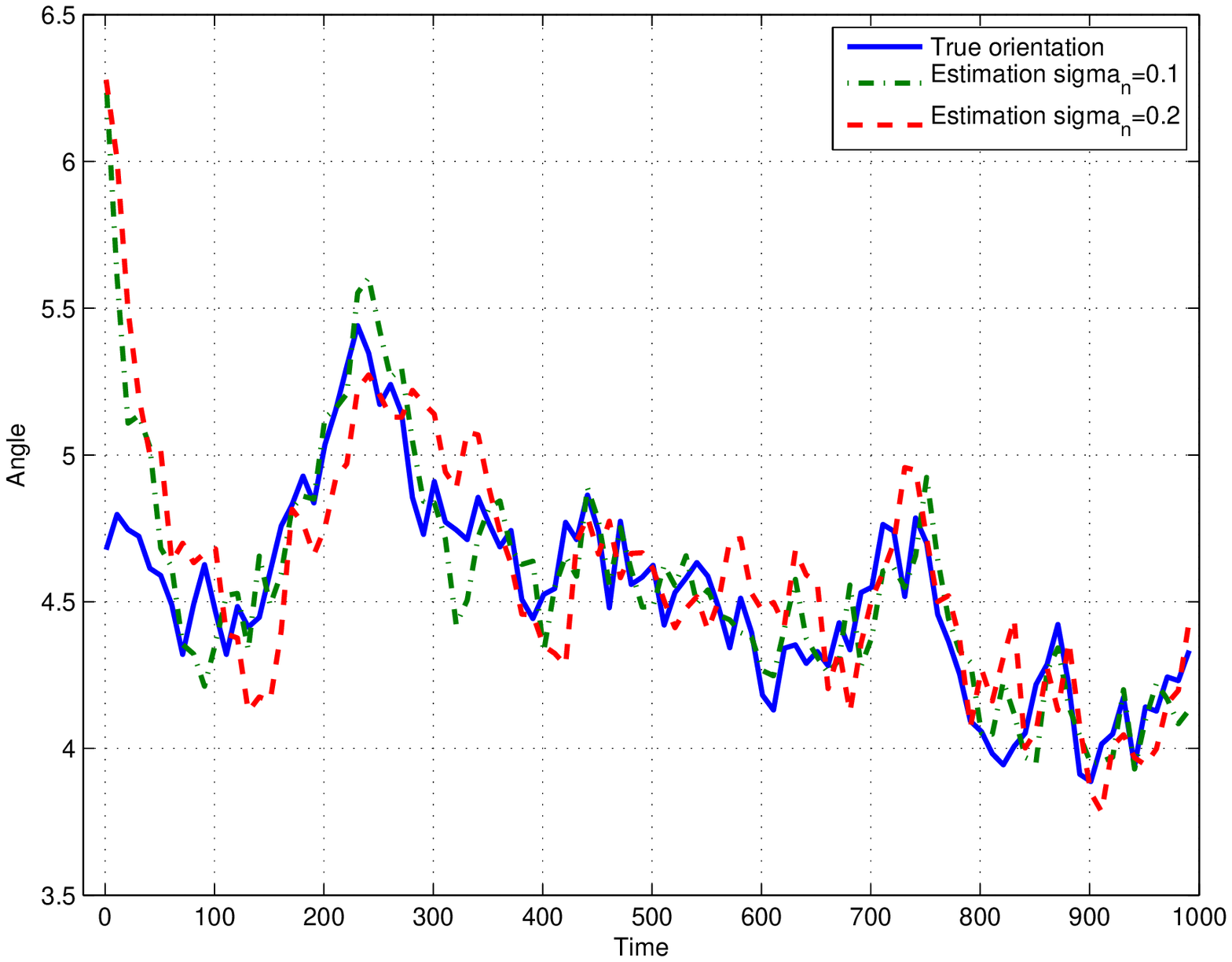}}
  \subfigure[]{\includegraphics[width=6.5cm]{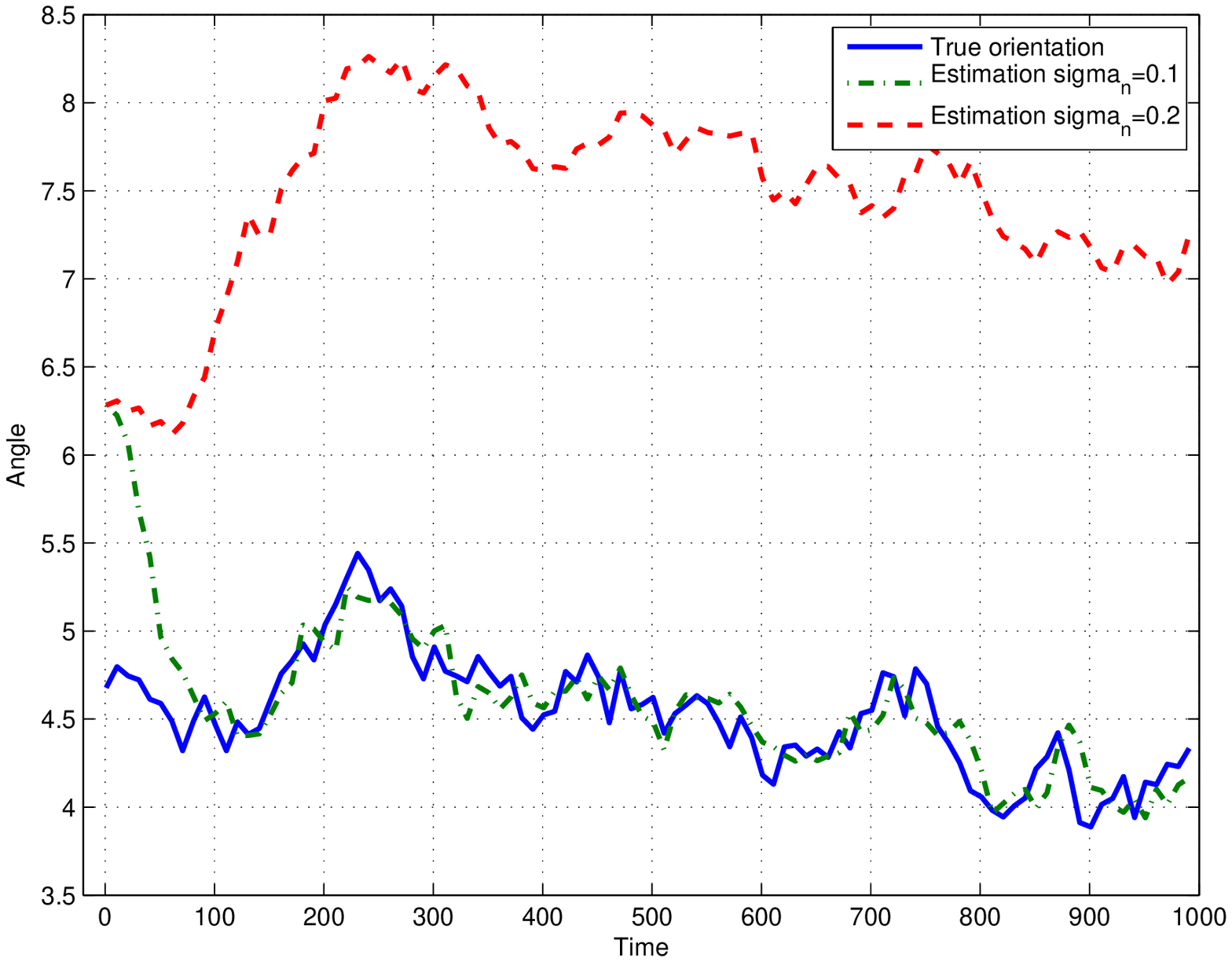}}
  %\vspace{7.5cm}
  \caption{Results of tracking using the configuration of Fig.~\ref{fig:target_path} at different
    noise levels.  First row: coordinate in $x$-axis. Second row: coordinate in $y$-axis. Last row:
    orientation.  In the first column the target has size $10$, while in the second column the
    target has size 1. The solid line always indicates the true system state.}
  \label{fig:tracking_big_small_target}
\end{figure}

\section{CGPT reconstruction and tracking problem in the limited-view setting}
\label{sec:limited_angle_view}

In this section we study the stability of CGPTs reconstruction and
tracking problem in the case $0<\gamma<2\pi$, always under the
condition that $N>2K$, {\it i.e.}, the number of sources/receivers
is two times larger than the highest order of CGPTs to be
reconstructed. Unlike in the full-view case, here $\bC$ is no
longer orthogonal in general, nonetheless one can still establish
the SVD of $\bL$ similarly as in Proposition
\ref{prop:full-angle-view-svd}.
\begin{prop}
  Consider the concentric and limited-view setting with $N\geq 2K$, and suppose that $\bC$ is
  of maximal rank. Let $\set{\mu_n}$ be the $n$-th largest eigenvalue of the matrix $\bD \bC^\top \bC \bD$ and
  let $\set{v_n}$ be the associated orthonormal eigenvector.
%Let $\set{\mu_n}$ be the $n$-th largest eigenvalue of the matrix $\bD \bC^\top \bC
  %\bD$ and $\set{v_n}$ be the associated orthonormal eigenvector.
  Then the $(a,b)$-th singular
  value of the operator $\bL$ is $\lambda_{ab}=\sqrt{\mu_a\mu_b}$, with the associated left singular
  vector the matrix $\mathbf{g}_{ab}=v_a v_b^\top$. In particular, the condition
  number of the operator $\bL$ is
  \begin{align}
    \label{eq:cond_L_lim_aov}
    \cond\bL = \cond{\bD\bC^\top\bC\bD} \leq \cond\bC^2
    K^2\rho^{2(K-1)},
  \end{align}
with $\cond\bC$ being the condition number of the matrix $\bC$.
\end{prop}
\begin{proof}
  %We note $\bC=USV^\top$ the SVD of $\bC$, and $W=VS^\top S V^\top$.
  We first note that for any matrices $\bU, \bV$ we have:
  \begin{align*}
    \seq{\bL(\bU), \bL(\bV)} = \seq{\bU, (\bD\bC^\top\bC\bD) \bV
    (\bD\bC^\top\bC\bD)}.
  \end{align*}
   Taking $\mathbf{g}_{ab}=v_av_b^\top$, and
  $\mathbf{g}_{a'b'}=v_{a'}v_{b'}^\top$, we get
  \begin{align*}
    \seq{\bL(\mathbf{g}_{ab}), \bL(\mathbf{g}_{a'b'})} &= \mu_{a'}\seq{v_av_b^\top, v_{a'}v_{b'}^\top(\bD
      \bC^\top\bC\bD)} = \mu_{a'}\mu_{b'}\seq{v_av_b^\top,
      v_{a'}v_{b'}^\top} \\
    &= \delta_{aa'}\delta_{bb'}\mu_a\mu_b,
  \end{align*}
 where $\delta_{aa'}$ is the Kronecker's symbol, which implies that $\norm{\bL(\mathbf{g}_{ab})}_F
 = \sqrt{\mu_a\mu_b}$ is the $(a,b)$-th singular value of
  $\bL$. We denote by $\rho_{\text{max}}(\cdot), \rho_{\text{min}}(\cdot)$ the maximal and the minimal
  singular values of a matrix, then
  \begin{align*}
    \rho_{\text{max}}(\bD \bC^\top \bC \bD) &= \rho_{\text{max}}(\bC \bD)^2
    \leq\rho_{\text{max}}(\bC)^2 \rho_{\text{max}}(\bD)^2,\\
    \rho_{\text{min}}(\bD \bC^\top \bC \bD) &= \rho_{\text{min}}(\bC
     \bD)^2 \geq\rho_{\text{min}}(\bC)^2 \rho_{\text{min}}(\bD)^2,
  \end{align*}
  and the condition number of $\bL$ is therefore bounded by $\cond\bC^2 K^2\rho^{2(K-1)}$.
\end{proof}

\subsection{Injectivity of $\bC$}
\label{sec:injectivity-c}

We denote by $V_K$ the vector space of functions of the form
\begin{align}
  \label{eq:fourier_serie_complex}
  f(\theta)= \sum_{k=-K}^K c_ke^{ik\theta},
\end{align}
with $c_k\in\C$, and $V_K^0$ the subspace of $V_K$ such that
$c_0=0$. Functions of $V^0_K$ can be written as
\begin{align}
  \label{eq:fourier_serie}
  f(\theta)=\sum_{k=1}^K
  \alpha_k\cos(k\theta)+\beta_k\sin(k\theta),
\end{align}
with $\alpha_k, \beta_k\in\C$.
% where $\set{\cos(k\cdot),\sin(k\cdot)}_k$ is an orthogonal basis.  \ignore{Then the function of type
%   \eqref{eq:fourier_serie} lives in a subspace of $V_K^0$.}
Observe that taking discrete samples of \eqref{eq:fourier_serie}
at $\theta_s=\gamma s/N$ is nothing but applying the matrix $\bC$
on a coefficient vector
$(\alpha_1,\beta_1\ldots\alpha_K,\beta_K)$. We have the following
result.
\begin{prop}
  For any $N\geq 2K$, the matrix $\bC$ is of maximal rank.% there exists a unique function $f$ of type
  % \eqref{eq:fourier_serie_complex} such that $f(\gamma s/N)=y_s$ for all $s=1\ldots N$.
\end{prop}
\begin{proof}
  % We need to show that if $f(\theta_s)=0$ for $s=1\lots N$, then $f(\theta)\equiv 0$ on $[0,2\pi)$.
  Multiplying $f\in V_K^0$ in \eqref{eq:fourier_serie_complex} by $e^{i K\theta}$, and using the
  fact that $c_0=0$, we have
  \begin{align}
    \label{eq:e_m_f}
    e^{i K\theta} f(\theta) &= \sum_{k=0}^{K-1} c_{k-K} e^{i k \theta} + \sum_{k=K+1}^{2K} c_{k-K}
    e^{i k \theta} \notag \\
    &=  \sum_{k=0}^{K-1} c_{k-K} e^{i k \theta} + \sum_{k=K}^{2K-1} e^{i\theta}c_{k+1-K} e^{i k
      \theta} = \sum_{k=0}^{2K-1} \tilde c_k e^{i k \theta},
  \end{align}
  where  $\tilde c_k=c_{k-K}$ for $k=0,\ldots, K-1$, and $\tilde c_k = e^{i\theta}
  c_{k+1-K}$ for $k=K, \ldots, 2K-1$. The $N$ vectors $v_s:=(e^{i k\theta_s})_{k=0\ldots 2K-1}$
  are
  linearly independent since they are the first $2K\leq N$ rows of a $N \times N$ Vandermonde
  matrix. Therefore, $f(\theta_s)=0$ for $s=1\ldots N$ implies that $\tilde
c_k=0$ for all $k=0,\ldots, 2K-1$,
 which means that $\bC$ is of maximal rank.
\end{proof}
Consequently, for arbitrary range $0<\gamma\leq 2\pi$, a
sufficient condition to uniquely determine the CGPTs of order up
to $K$ is to have $N \geq 2K$ sources/receivers.

\subsection{Explicit left inverse of $\bC$}
\label{sec:explicit_linv_C}

We denote by $D_K(\theta)$ the Dirichlet kernel of order $K$:
\begin{align}
  \label{eq:Dirichlet_kernel}
  D_K(\theta)=\sum_{k=-K}^K e^{ik\theta} =
  \frac{\sin((K+1/2)\theta)}{\sin(\theta/2)}.
\end{align}
We state without proof the following well known result about
$V_K$.
\begin{lem}
\label{lem:VK_innerprod_sum} The functions
$\set{D_K(\theta-\frac{2\pi n}{2K+1})}_{n=0,\ldots, 2K}$ is an
orthogonal basis of $V_K$. For any $f,g\in V_{K}$, the following
identity holds:
\begin{align}
  \label{eq:VK_innerprod_sum}
  \frac {1}{2\pi} \int_0^{2\pi} f(\theta) g^*(\theta)d\theta = \frac 1 {2K+1}\sum_{n=1}^{2K+1}f\Paren{\frac {2\pi
      n}{2K+1}}g\Paren{\frac {2\pi n}{2K+1}},
\end{align}
where $^*$ denotes the complex conjugate. In particular,  we have
for $n=0, \ldots, 2K$
\begin{align}
  \label{eq:VK_dirichlet_coeff}
  \frac {1}{2\pi} \int_0^{2\pi} f(\theta)D_K\Paren{\theta-\frac {2\pi n}{2K+1}}
  d\theta = f\Paren{\frac {2\pi n}{2K+1}}.
\end{align}
\end{lem}
\begin{comment}
  \begin{proof}
    It is well known that any $f\in V_K$ can be uniquely reconstructed using $N>2K$ equally spaced samples:
    \begin{align}
      \label{eq:periodic_recon_2pi}
      f(t) = \frac 1 N \sum_{n=0}^{N-1}
      f\Paren{\tn}D_K\Paren{t-\tn}.
    \end{align}
    Using the property of the Dirichlet kernel:
    \begin{align*}
      \frac {1}{2\pi} \int_0^{2\pi} D_K\Paren{t-\frac {2\pi n}{N}}D_K^*\Paren{t-\frac
      {2\pi n'}{N}} dt = D_K\Paren{\frac {2\pi(n-n')}{N}},
    \end{align*}
    we verify easily \eqref{eq:VK_innerprod_sum} and \eqref{eq:VK_dirichlet_coeff} in the case $N=2K+1$.
  \end{proof}
\end{comment}

\begin{lem}
\label{lem:VK_interpolation}
  Given a set of $N>2K$ different points $0< \theta_1<\ldots<\theta_N\leq 2\pi$, there exist
  interpolation kernels $h_s\in V_{\floor{N/2}}$ for $s=1\ldots N$, such that:
  \begin{align}
    \label{eq:f_interpl}
    f(\theta) = \sum_{s=1}^{N} f(\theta_s) h_s(\theta) \ \text{ for any } f\in
    V_K.
  \end{align}
  %In particular, these kernels are uniquely defined when $N$ is odd.
\end{lem}
\begin{proof}
%  The form of $h_s$ depends on $N$.
  When the number of points $N$ is odd, it is well known \cite{zygmund_trigonometric_1988}
  that $h_s$ takes the form
  \begin{align}
    \label{eq:hs_odd}
    h_s(\theta) = \prod_{t=1,t\neq
      s}^{N}\frac{\sin\Paren{\frac{\theta-\theta_t}{2}}}{\sin\Paren{\frac{\theta_s-\theta_t}{2}}}.
  \end{align}
  When $N$ is even, by a result established in
 \cite{margolis_nonuniform_2008}
  \begin{align}
    \label{eq:hs_even}
    h_s(\theta) = \cos\Paren{\frac{\theta-\theta_s}{2}}\prod_{t=1,t\neq
      s}^{N}\frac{\sin\Paren{\frac{\theta-\theta_t}{2}}}{\sin\Paren{\frac{\theta_s-\theta_t}{2}}}.
  \end{align}
  It is easy to see that in both cases $h_s$ belongs to $V_{\floor{N/2}}$.
  % and \eqref{eq:f_interpl} in this case is a result established in \cite{margolis_nonuniform_2008}.
  % which is an interpolation kernel, and belongs to $V_{\floor{N/2}}$. For any $f\in \VfN$, the
  % function
  % \begin{align}
  %   \label{eq:tilde_f_interpl}
  %   \tilde f(\theta) = \sum_{s=1}^{N} f(\theta_s) h_s(\theta)
  % \end{align}
  % also belongs to $V_{\floor{N/2}}$ and takes the same value as $f$ at the points
  % $\theta_1\ldots\theta_N$. Therefore $\tilde f - f$ is identically zero since any non zero function
  % of $V_{\floor{N/2}}$ can have at most $2{\floor{N/2}}$ roots, and we get
  % \eqref{eq:f_interpl}.
  % % Suppose that there exists another collection of kernels $\tilde{h_s}\in\VfN$ also satisfying
  % % \eqref{eq:f_interpl}, then we have
  % % \begin{align*}
  % %   h_s(\theta) = \sum_{t=1}^N h_s(\theta_t)\tilde{h_s}(\theta) = \tilde{h_s}(\theta), \ \forall
  % %   s=1\ldots N
  % % \end{align*}
\end{proof}
% \begin{rmk}
%   The special forms \eqref{eq:hs_odd} and \eqref{eq:hs_even} both satisfy $h_s(\theta_t)=1$ if
%   $s=t$, and $h_s(\theta_t)=0$ if $s\neq t$. We did not require this for the interpolation kernel,
%   which makes that for $N>2K$ there can exists $h_s$ other than \eqref{eq:hs_odd} and
%   \eqref{eq:hs_even}.
% \end{rmk}

Now we can find explicitly a left inverse for $\bC$.
\begin{prop}
  \label{prop:expl-left-inverse}
  Under the same condition as in Lemma \ref{lem:VK_interpolation}, we denote by $h_s$ the interpolation
  kernel and define the matrix
  %Let $N>2K$ and $h_s\in \VfN$ for $s=1\ldots N$ be interpolation kernels. We define the matrix
  $\tilde \bC=(\tilde \bC_{ks})_{k,s}$ as
  \begin{align}
    \label{eq:leftinv_C}
    \tilde \bC_{2k-1,s} = \frac{1}{\pi}\int_0^{2\pi} h_s(\theta) \cos(k\theta) d\theta, \ \
    \tilde \bC_{2k,s} = \frac{1}{\pi}\int_0^{2\pi} h_s(\theta) \sin(k\theta)
    d\theta.
  \end{align}
  Then $\tilde \bC \bC = \bI$.  In particular, if $N$ is odd, the matrix $\tilde \bC$ can be
  calculated as
  \begin{align}
    \label{eq:leftinv_C_sum}
    \tilde \bC_{2k-1,s} = \frac{2}{N} \sum_{n=1}^{N}h_s\Paren{\frac {2\pi n}{N}}
    \cos\Paren{\frac{2\pi kn}{N}}, \ \
    \tilde \bC_{2k,s} = \frac{2}{N} \sum_{n=1}^{N}h_s\Paren{\frac {2\pi n}{N}} \sin\Paren{\frac{2\pi
    kn}{N}}.
  \end{align}
\end{prop}

\begin{proof}
  Given $v=(\alpha_1,\beta_1\ldots\alpha_K,\beta_K)\in\C^{2K}$, and $f$ the associated function
  defined by \eqref{eq:fourier_serie}, we have $(\bC v)_n = f(\theta_n)$ for $n=1,\ldots, N$.
  Using \eqref{eq:f_interpl}
  and \eqref{eq:leftinv_C}, we find that
  \begin{align}
    \label{eq:CCv}
    (\tilde \bC\bC v)_{2k-1} &= \frac{1}{\pi}\int_0^{2\pi} f(\theta) \cos(k\theta) d\theta = \alpha_k, \\
    (\tilde \bC\bC v)_{2k} &= \frac{1}{\pi}\int_0^{2\pi} f(\theta) \sin(k\theta) d\theta =
    \beta_k,
  \end{align}
 and therefore, $\tilde \bC \bC v= v$.  Observe that $h_s(\theta)$, $\cos(k\theta),$ and $\sin(k\theta)$
  all belong to $V_{\floor{N/2}}$, so when $N$ is odd, we  easily deduce \eqref{eq:leftinv_C_sum}
  using \eqref{eq:VK_innerprod_sum}.
\end{proof}

\begin{rmk}
  In general, the left inverse $\tilde\bC$ in \eqref{eq:leftinv_C} is not the pseudo-inverse of
  $\bC$, and by definition, we have $\pinv\bC=\tilde\bC$ if $\bC\tilde\bC$ is symmetric.
  If
  $P_{V_K^0}(h_s)$ is the orthogonal projection of $h_n$ onto
  $V_K^0$, {\it i.e.},
  \begin{align}
    \label{eq:hs_VK_proj}
    P_{V_K^0}(h_s)(\theta) = \sum_{k=1}^K  \tilde \bC_{2k-1,s}\cos(k\theta) +  \tilde
    \bC_{2k,s}\sin(k\theta),
  \end{align}
  then, $P_{V_K^0}(h_s)(\theta_t) = (\bC\tilde \bC)_{st}$. Therefore, $\tilde\bC$ is the pseudo-inverse of $\bC$ if and only if the interpolation kernel $h_s$
satisfies:
  \begin{align}
    \label{eq:sampling_kernel_cond}
    P_{V_K^0}(h_s)(\theta_t) = P_{V_K^0}(h_t)(\theta_s), \ \for s,t=1\ldots
    N.
    % P_{V_K}(h_s)(\theta_t) - \int_0^{2\pi}h_s(\theta)d\theta = P_{V_K}(h_t)(\theta_s) -
    % \int_0^{2\pi}h_t(\theta)d\theta, \ \forall s,t=1\ldots N
  \end{align}
  % thus \eqref{eq:sampling_kernel_cond}.  For this, we need unless under special conditions.
  % Furthermore, where $P_{V_K^0}$ denotes the orthogonal projection onto $V_K^0$.
\end{rmk}

%\graphicspath{{../figures/cond_lim_aov/}}

\begin{rmk} Proposition \ref{prop:expl-left-inverse} can be used
in the noiseless limited-view case to reconstruct the CGPT matrix
$\bM$ from the MSR measurements $\bV$. In fact, from
(\ref{eq:op_L}) it immediately follows that
$$
  \bM = \bD^{-1} \tilde{\bC} \bV \tilde{\bC}^\top \bD^{-1}. $$
This shows that in the noiseless case, the limited-view aspect has
no effect on the reconstruction of the GPTs, and consequently on
the location and orientation tracking. In the presence of noise,
the effect, as will be shown in the next subsection, is dramatic.
A small amount of measurement noise significantly changes the
performance of our algorithm unless the arrays of receivers and
transmitters offer a directional diversity, see
Fig.~\ref{fig:target_path_lim_aov}.
\end{rmk}

\subsection{Ill-posedness in the limited-view setting}
\label{sec:recon_cgpt_num}
% The injectivity of $\bC$ and the left inverse $\tbC$ do not guarantee the stability of inversion
% procedure.
We undertake a numerical study to illustrate the ill-posedness of
the linear system \eqref{eq:MSR_CGPT_linsys} in the case of
limited-view data. Fig.~\ref{fig:svd_CtC_DCtCD_gamma} shows the
distribution of eigenvalues of the matrix $\CtC$ and $\DCtCD$ at
different values of $\gamma$ with $N=101$ and $K=50$. In
Fig.~\ref{fig:svd_CtC_DCtCD_cond}, we calculate the condition
number of $\CtC$ and $\bL$ (which is equal to that of $\DCtCD$ by
\eqref{eq:cond_L_lim_aov}) for different orders $K$. From these
results, we see clearly the effect of the limited-view aspect.
First, the tail of tiny eigenvalues in
Fig.~\ref{fig:svd_CtC_DCtCD_gamma}.(a) suggests that the matrix
$\CtC$ is numerically singular, despite the fact that $\bC$ is of
maximal rank. Secondly, both $\CtC$ and $\bL$ rapidly become
extremely ill-conditioned as $K$ increases, so the maximum
resolving order of CGPTs is very limited. Furthermore, this limit
is intrinsic to the angle of view and cannot be improved by
increasing the number of source/receivers, see
Fig.~\ref{fig:svd_CtC_DCtCD_cond} (c) and (d).

\begin{figure}[htp]
  \centering
  \subfigure[Eigenvalues of $\CtC$]{\includegraphics[width=7.5cm]{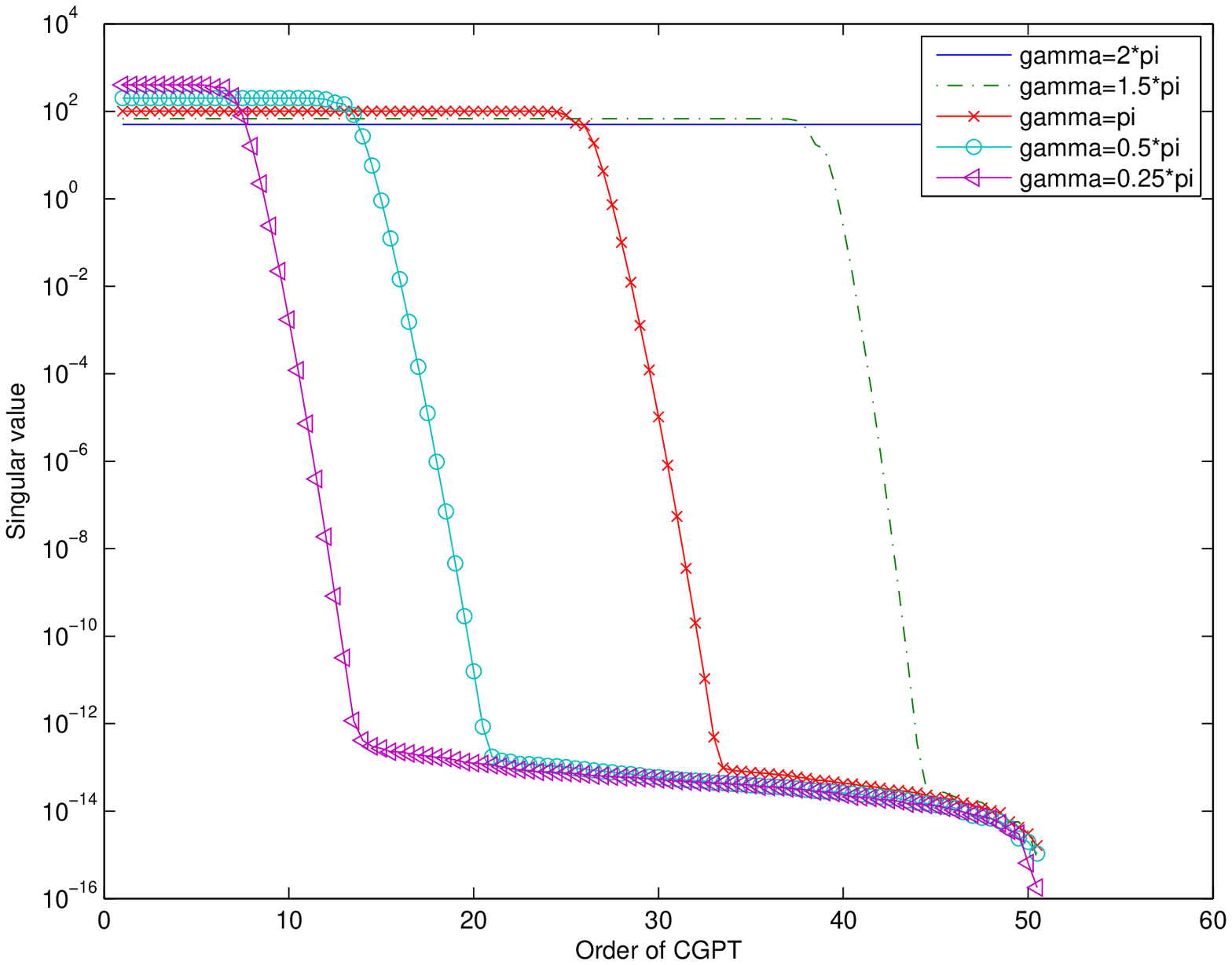}}
  \subfigure[Eigenvalues of $\DCtCD$]{\includegraphics[width=7.5cm]{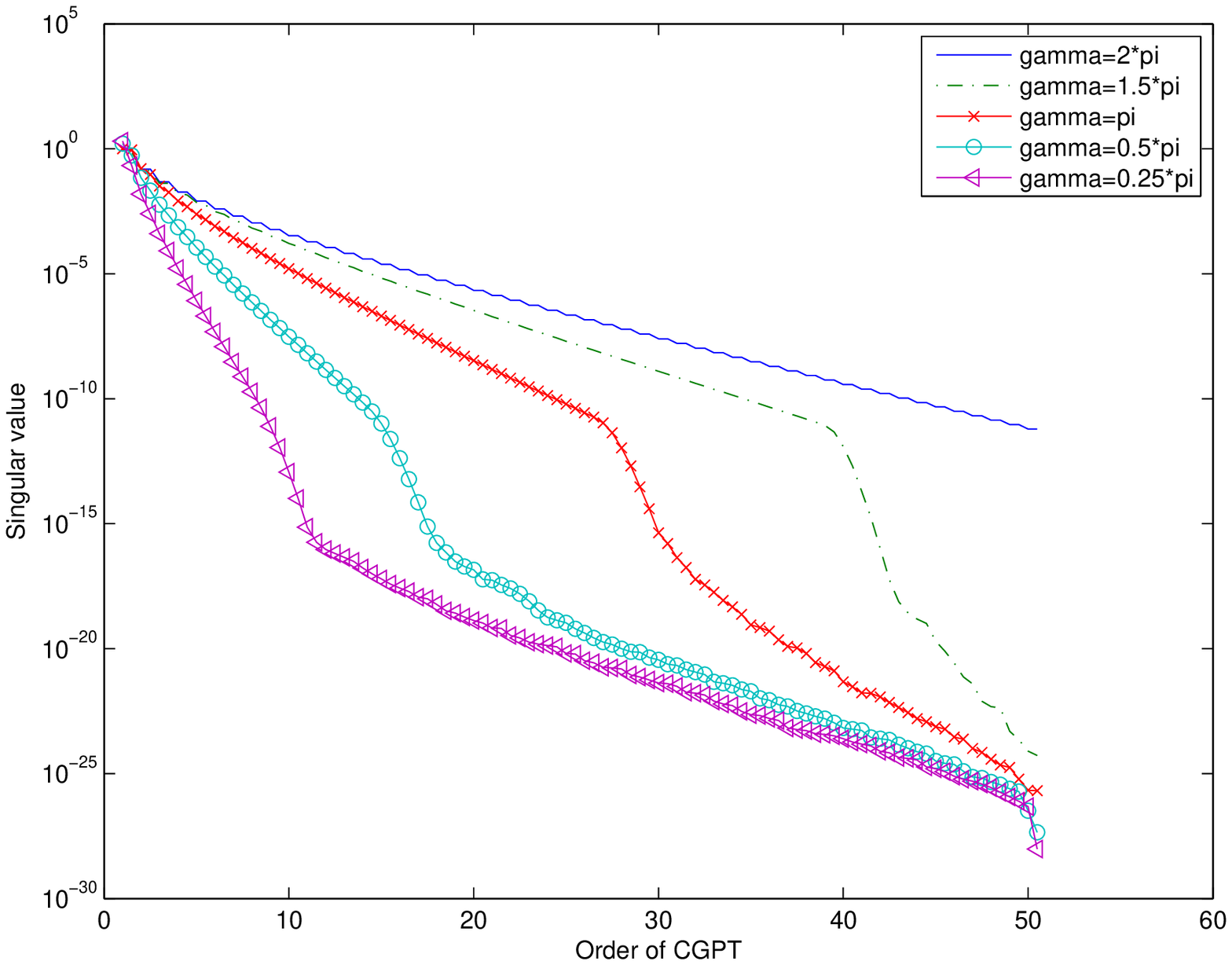}}
  \caption{Distribution of eigenvalues (in log scale) of the matrix $\CtC$ (a) and $\DCtCD$
    (b). $N=101$ sources are equally spaced between $[0,\gamma)$ on a circle of radius $\rho=1.2$,
    and $K=50$. Each curve corresponds to a different value of $\gamma$. The matrix $\CtC$ and
    $\DCtCD$ are calculated from these parameters and their eigenvalues are sorted in decreasing
    order.}
  \label{fig:svd_CtC_DCtCD_gamma}
\end{figure}

\begin{figure}[htp]
  \centering
  \subfigure[Condition number of $\CtC$]{\includegraphics[width=7.5cm]{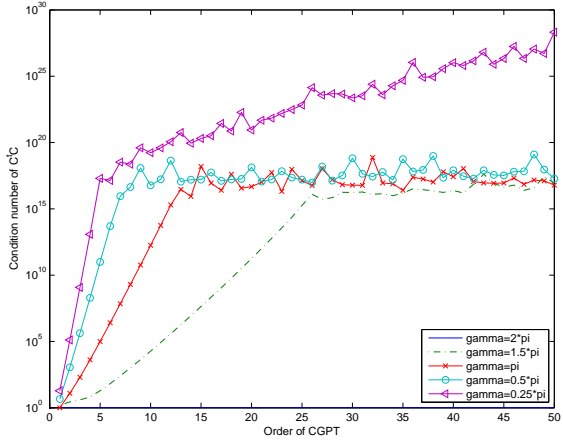}}
  \subfigure[Condition number of $\bL$]{\includegraphics[width=7.5cm]{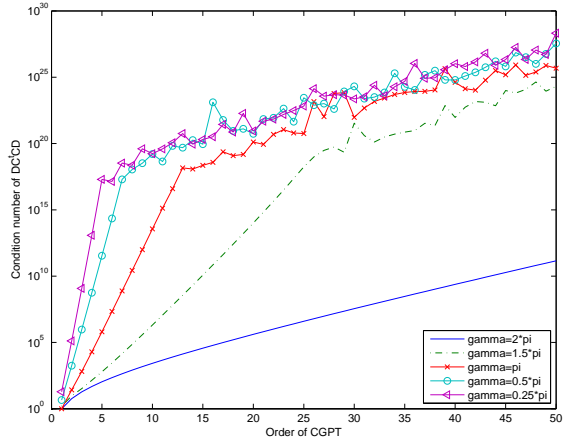}}
  \subfigure[Condition number of $\CtC$]{\includegraphics[width=7.5cm]{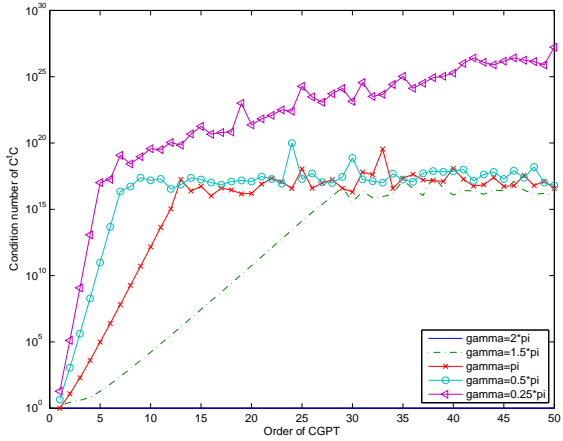}}
  \subfigure[Condition number of $\bL$]{\includegraphics[width=7.5cm]{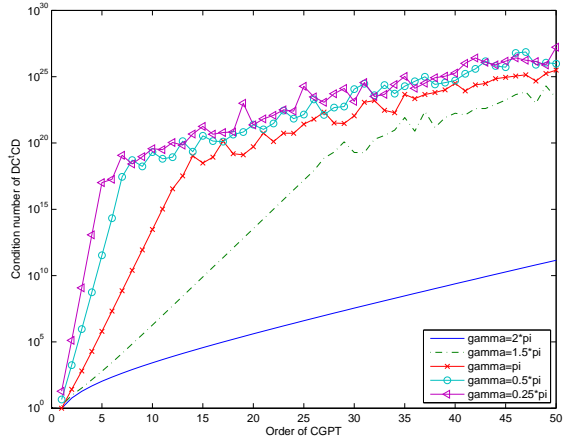}}
  \caption{Condition numbers (in log scale) of the matrix $\CtC$ (a) and the operator $\bL$ (b)
  for
    different orders $K$ between $[1,50]$. As in Fig.~\ref{fig:svd_CtC_DCtCD_gamma}, $N=101$ sources
    are equally spaced between $[0,\gamma)$ on a circle of radius $\rho=1.2$. Fig.(c) and (d) are
    the same experiment as Fig.(a) and (b) but with $N=1001$.}
  \label{fig:svd_CtC_DCtCD_cond}
\end{figure}

\subsection{Reconstruction of CGPTs}
\label{sec:recon_cgpt_num2}

The analysis above suggests that the least-squares problem
\eqref{eq:Least_square_CGPT_recon} is not adapted to the CGPT
reconstruction in a limited-view setting. Actually, the truncation
error or the noise of measurement will be amplified by the tiny
singular values of $\bL$, and yields extremely instable
reconstruction of high-order CGPTs, {\it e.g.}, $K\geq 2$.
Instead, we, in order to reconstruct CGPTs from the MSR data, use
Thikhonov regularization and propose to solve
\begin{align}
  \label{eq:Least_square_reg_CGPT_recon}
    \min_{\bM}\  \norm{\bL(\bM) - \bV}_F^2 + \mu \norm{\bM}_F^2,
\end{align}
with $\mu>0$ a small regularization constant. It is well known
that the effect of the regularization term is to truncate those
singular values of $\bL$ smaller than $\mu$, which consequently
stabilizes the solution. The optimal choice of $\mu$ depends on
the noise level, and here we determine it from the range
$[10^{-6}, 10^{-1}]$ by comparing the solution of
\eqref{eq:Least_square_reg_CGPT_recon} with the true CGPTs.

Here we reconstruct the CGPTs of an ellipse with the parameter
$N=101, K=50$, and $\gamma$ varying between 0 and $2\pi$.  The
major and minor axis of the ellipse are 1 and 0.5 respectively. In
Fig.~\ref{fig:cgpt_lim_aov} we show the error of the first 2 order
CGPTs reconstructed through \eqref{eq:Least_square_reg_CGPT_recon}
and \eqref{eq:Least_square_CGPT_recon} at three different noise
levels. It can be seen that, for small $\gamma$, the error
obtained by \eqref{eq:Least_square_reg_CGPT_recon} is
substantially smaller.

%\graphicspath{{../figures/CGPT_lim_aov/}}
\begin{figure}[htp]
  \centering
  \subfigure[]{\includegraphics[width=7.5cm]{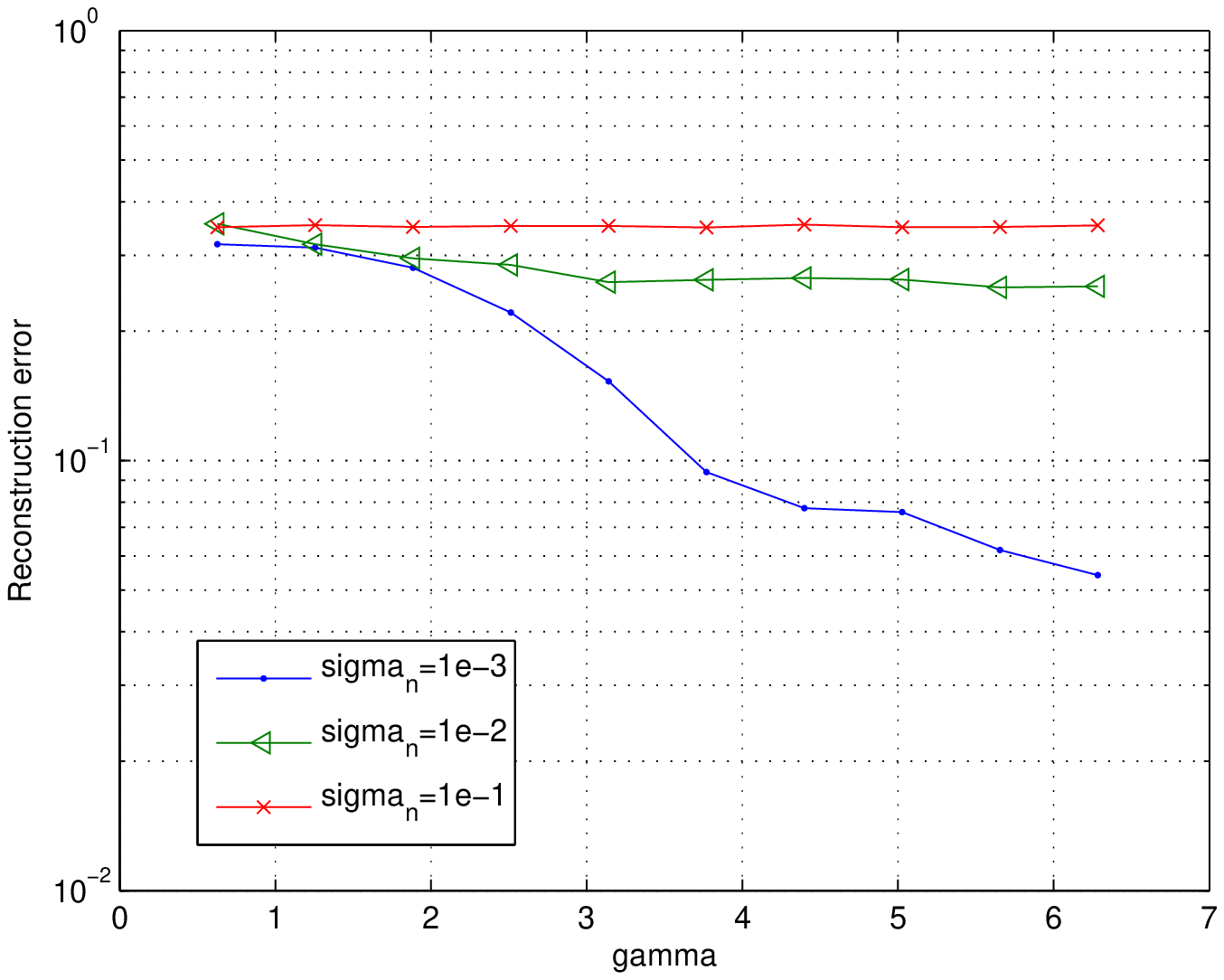}}
  \subfigure[]{\includegraphics[width=7.5cm]{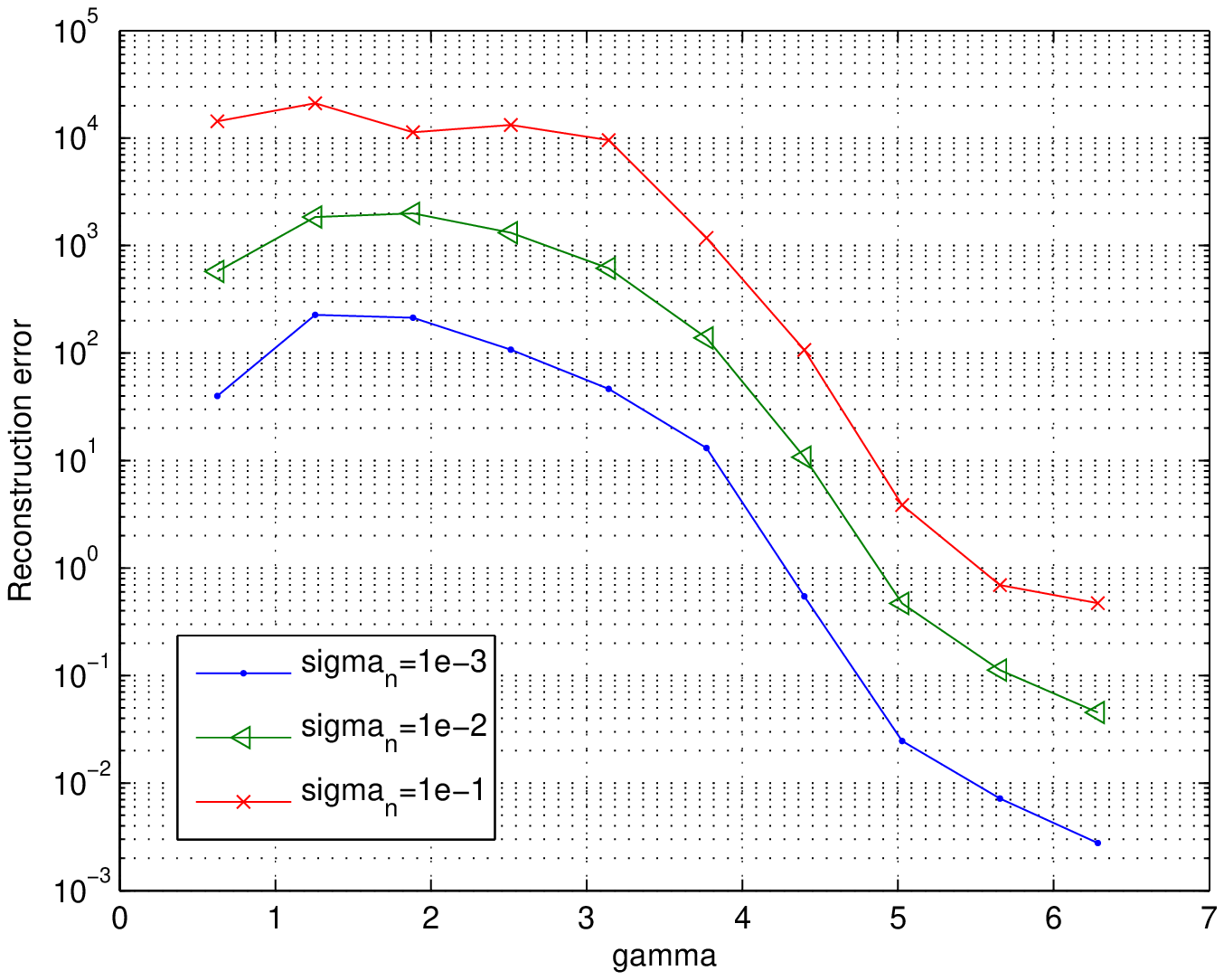}}
  \caption{Error of reconstructed CGPT of an ellipse compared with true CGPT values at different
    noise levels. We solve \eqref{eq:Least_square_reg_CGPT_recon} and
    \eqref{eq:Least_square_CGPT_recon} with $N=101, K=50$, and compare the first two orders with the
    true CGPT. The $x$-axis is the angle of view $\gamma$. Fig.(a): results of
    \eqref{eq:Least_square_reg_CGPT_recon}, Fig.(b): results of \eqref{eq:Least_square_CGPT_recon}.}
  \label{fig:cgpt_lim_aov}
\end{figure}

\subsection{Tracking in the limited-view setting}
The performance of the tracking algorithm can also be affected by
the limited angle of view. We repeat the experiment of subsection
\ref{sec:tracking_numexp_full_view} with $\delta=10$,
$\gamma=\pi$, and the same initial guess. In the first
configuration, $N=21$ sources/receivers are equally distributed
between $[0,\gamma)$, see Fig.~\ref{fig:target_path_lim_aov} (a).
The results of tracking by EKF presented in
Fig.~\ref{fig:tracking_lim_aov_uni_nonuni} (a), (c) and (e) show
large deviations in the estimation of position, and a totally
wrong estimation of orientation.  In the second configuration, we
divide the sources/receivers into 5 groups placed in a nonuniform
way on $[0, 2\pi)$, and each group covers only an angle range of
$0.2\pi$, see Fig.~\ref{fig:target_path_lim_aov} (b). Although the
total angular coverages are the same in both configurations, the
second one gives much better tracking results, as shown in
Fig.~\ref{fig:tracking_lim_aov_uni_nonuni} (b), (d) and (f). These
results clearly demonstrates the importance of a large angle of
view (or a directional diversity) for the tracking problem.

%\graphicspath{{../figures/tracking_lim_aov/}}
\begin{figure}[htp]
  \centering
  \includegraphics[width=7cm]{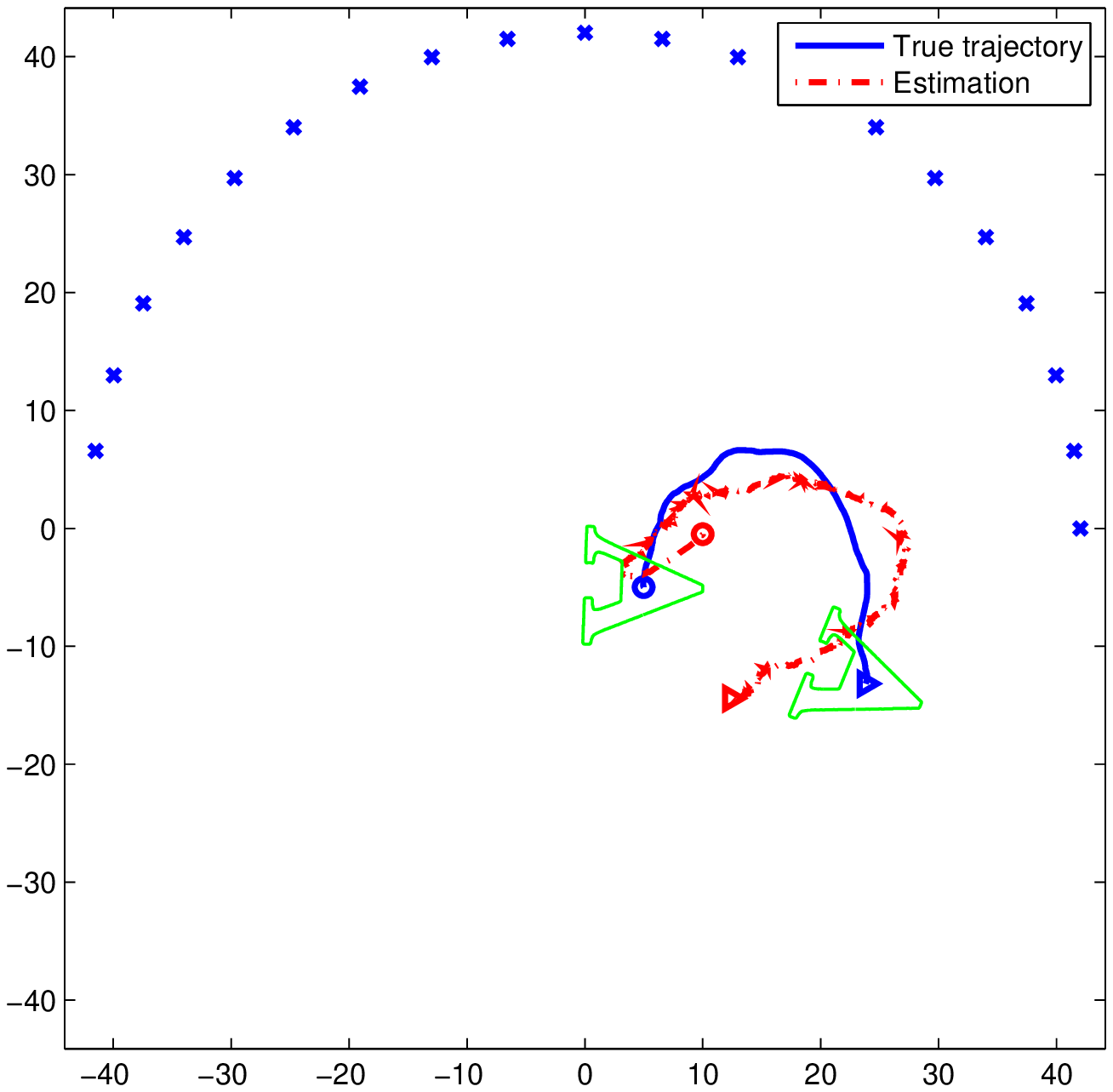}
  \includegraphics[width=7cm]{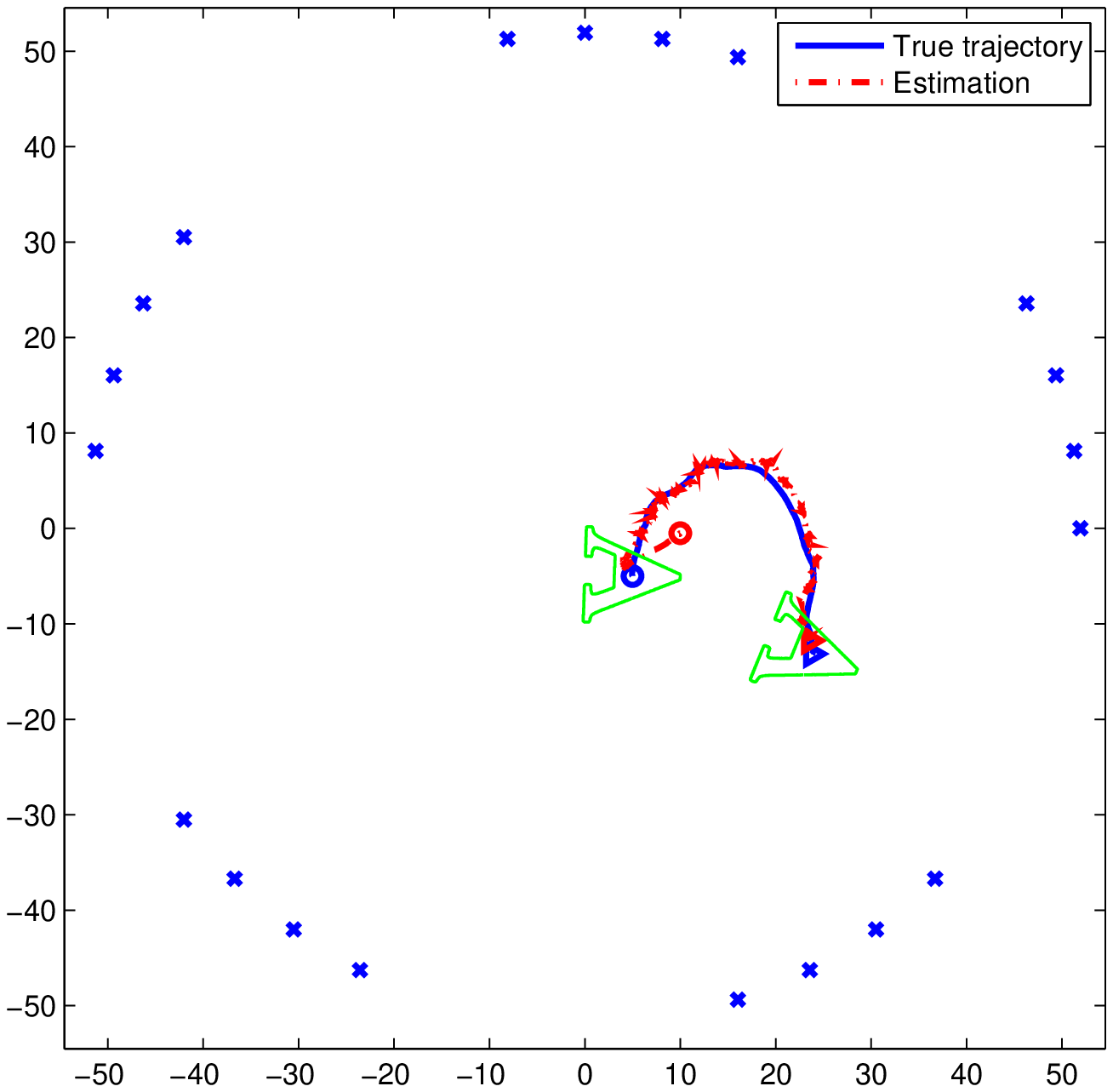}
  \caption{Same experiment as in Fig. \ref{fig:target_path}, with a limited angle of view
    $\gamma=\pi$. In Fig.(a) sources/receivers are equally distributed between $[0,\gamma)$, while
    in Fig.(b) they are divided into 5 groups.}
  \label{fig:target_path_lim_aov}
\end{figure}

\begin{figure}[htp]
  \centering
  \subfigure[]{\includegraphics[width=6.5cm]{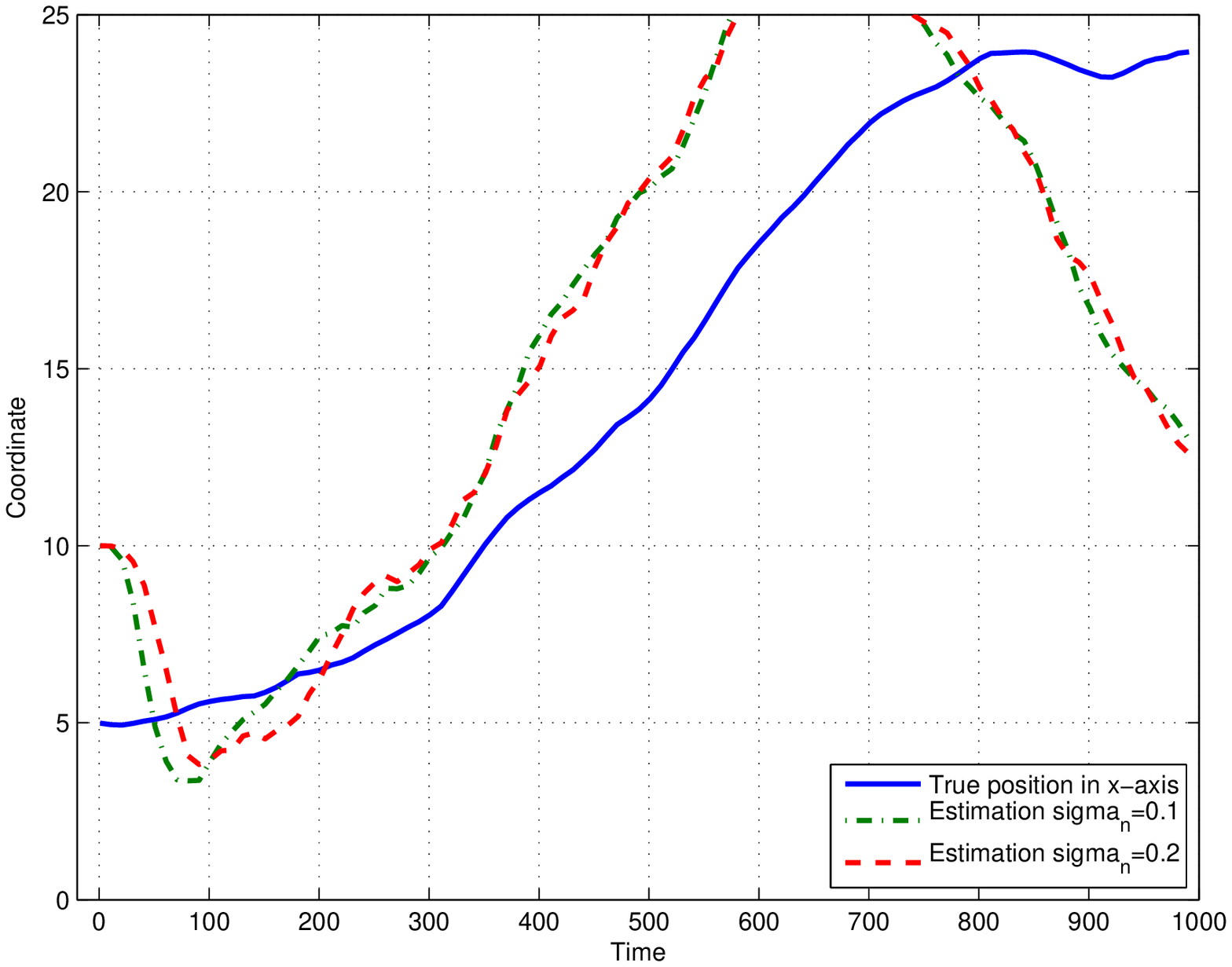}}
  \subfigure[]{\includegraphics[width=6.5cm]{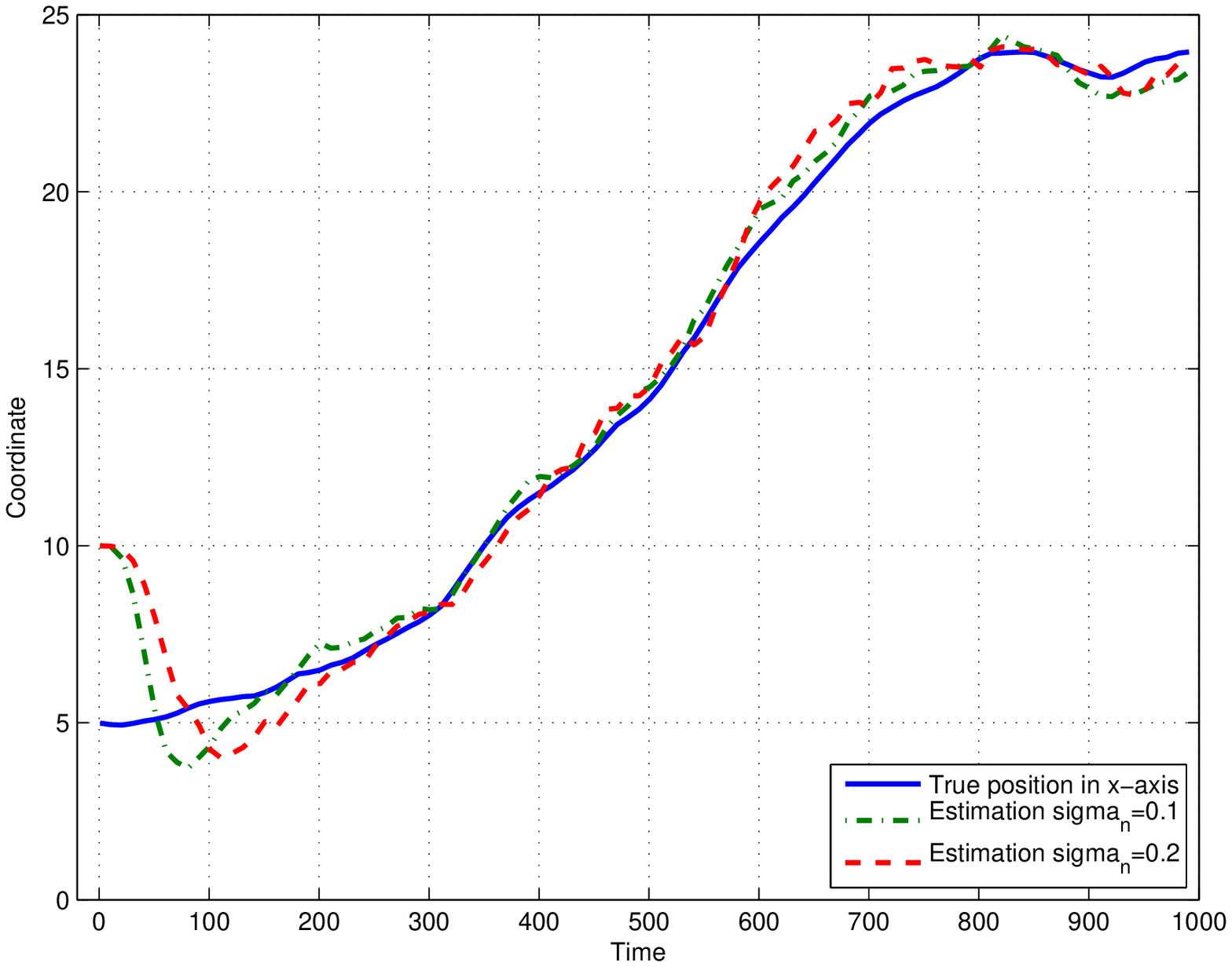}}
  \subfigure[]{\includegraphics[width=6.5cm]{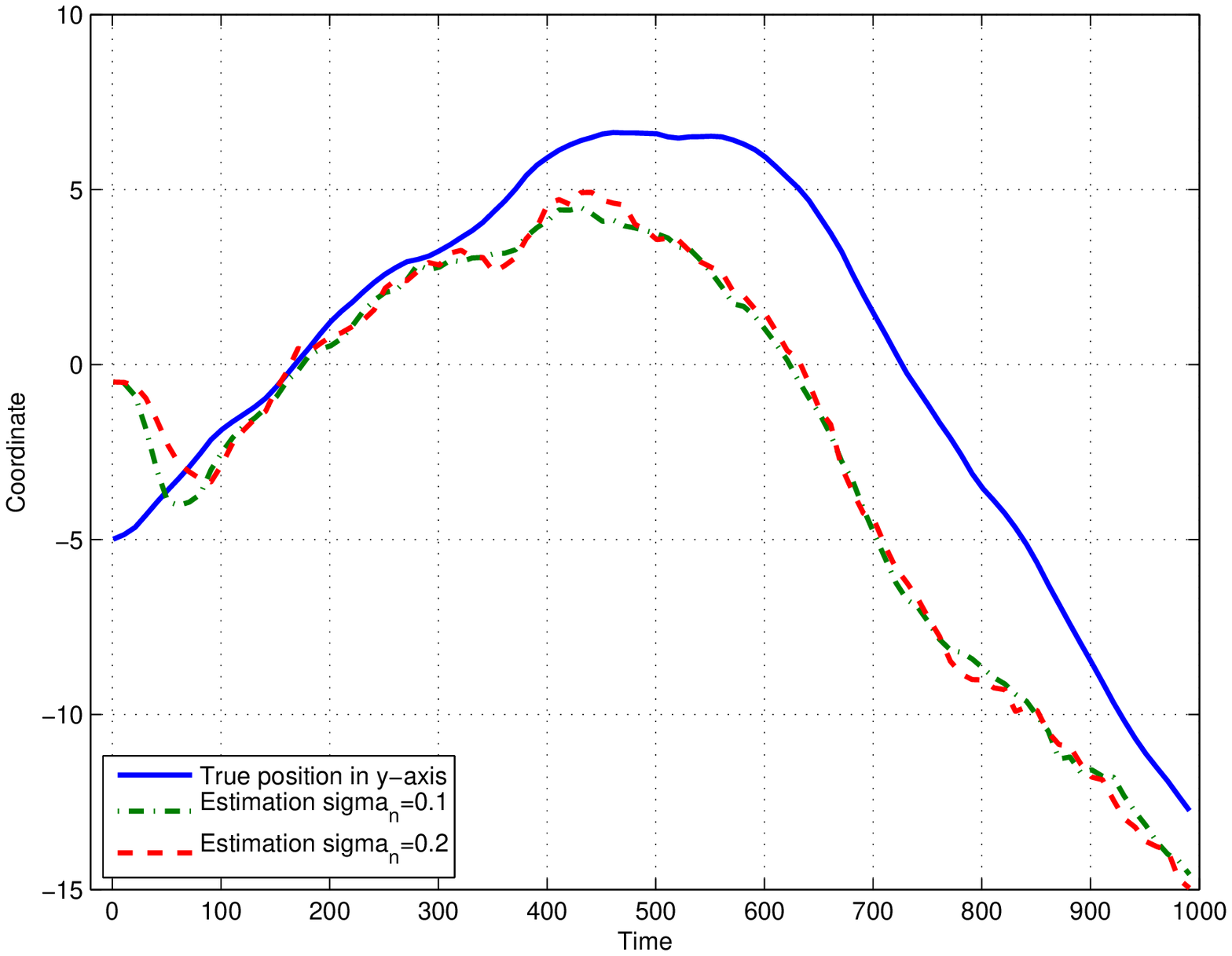}}
  \subfigure[]{\includegraphics[width=6.5cm]{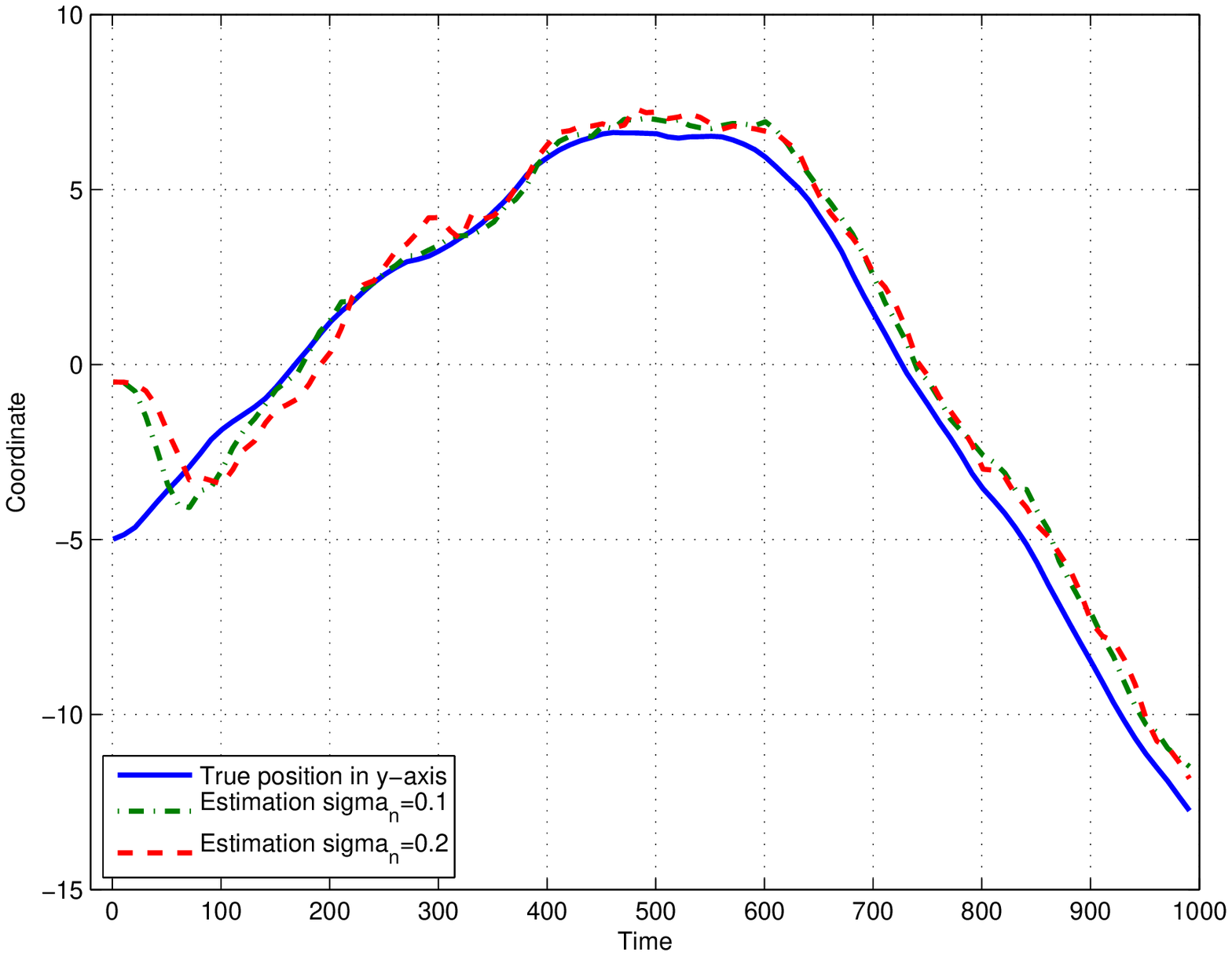}}
  \subfigure[]{\includegraphics[width=6.5cm]{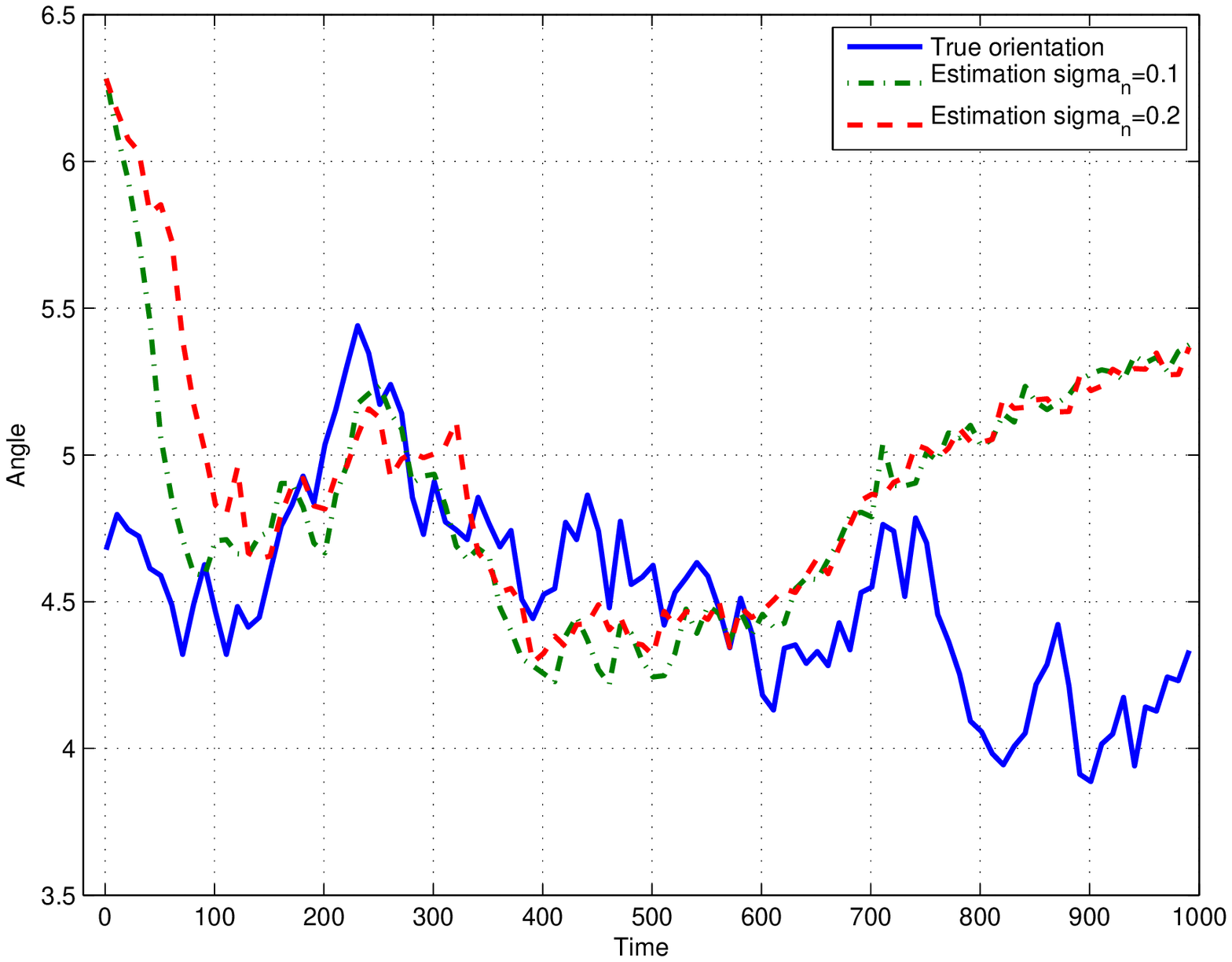}}
  \subfigure[]{\includegraphics[width=6.5cm]{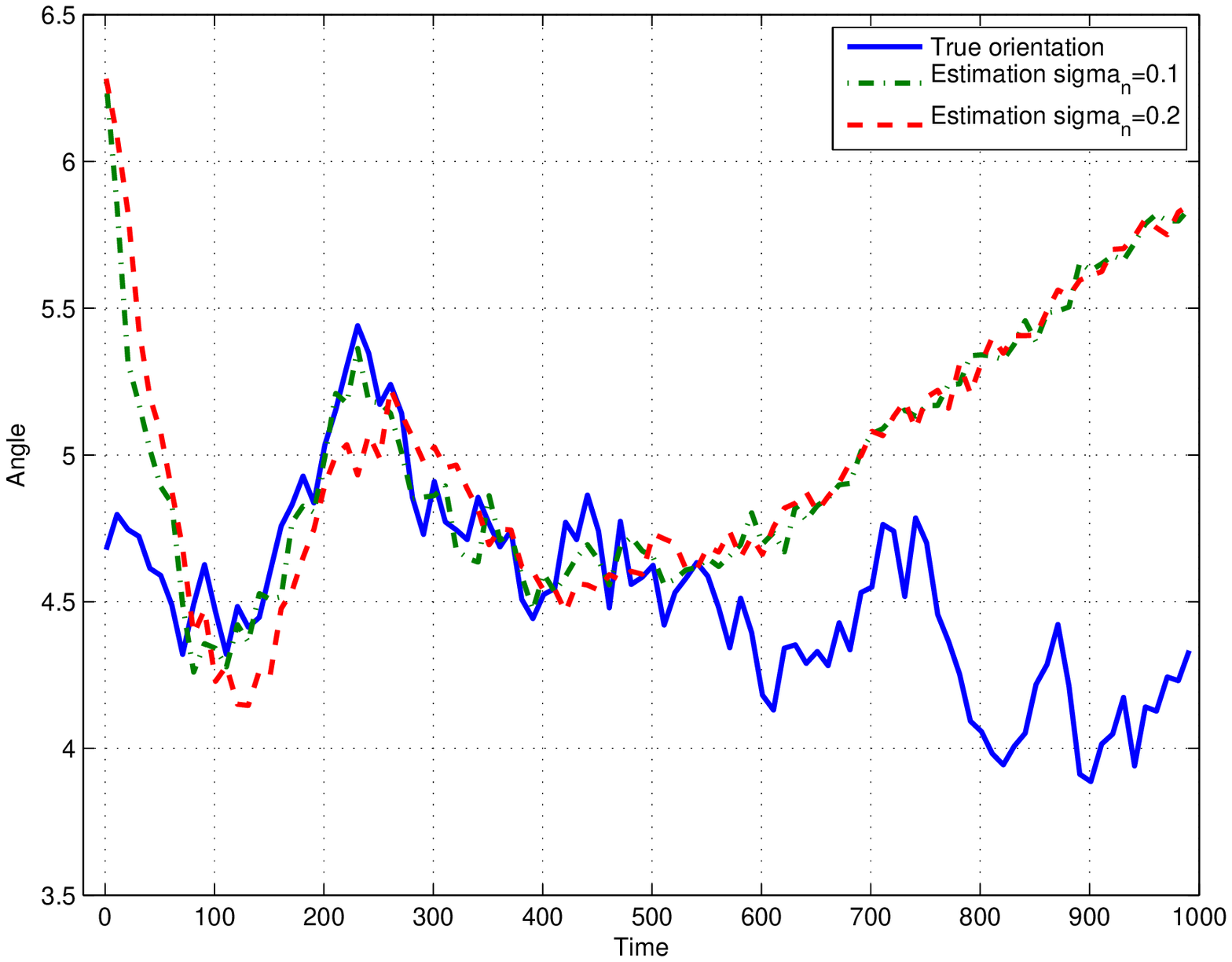}}
  \caption{Results of tracking using the configuration of Fig.~\ref{fig:target_path_lim_aov} at
    different noise levels. First row: coordinate in $x$-axis. Second row: coordinate in
    $y$-axis. Last row: orientation. First and second column correspond to the configuration in
    Fig.~\ref{fig:target_path_lim_aov} (a) and (b), respectively.}
  \label{fig:tracking_lim_aov_uni_nonuni}
\end{figure}

\section{Conclusion}
In this paper we have provided a location and orientation tracking of a mobile target
from MSR measurements in the full- and limited-view settings. Our algorithm is based
on the concept of GPTs. In the limited-view case, the effect of noise is severe on
the tracking.  However, if the arrays of receivers and transmitters offer a good
directional diversity, then satisfactory results can be obtained. It would be
interesting to generalize our algorithms for tracking multiple targets. As a first
step, a matching pursuit algorithm \cite{mallat} would be appropriate for recognizing
the targets. This will be the subject of a forthcoming work.

\appendix
\section{Kalman Filter}
\label{sec:KalmanFilter} The KF is a recursive method that uses a
stream of noisy observations to produce an optimal estimator of
the underlying system state \cite{kalman}. Consider the following
time-discrete dynamical system ($t\geq 1$):
\begin{align}
  \label{eq:Dynamic_system}
  X_t &= F_tX_{t-1}+W_t , \\ %\label{eqn:Kalman_eq_sys}
  Y_t &= H_tX_t+V_t .%\label{eqn:Kalman_eq_obs}
\end{align}
where
\begin{itemize}
\item $X_t$ is the vector of \emph{system state}; \item $Y_t$ is
the vector of \emph{observation}; \item $F_t$ is the state
transition matrix which is applied to the previous state
  $X_{t-1}$;
\item $H_t$ is the observation matrix which yields the (noise
free) observation from
  a system state $X_t$;
\item $W_t\sim\normallaw{Q_t}$ is the process noise and
$V_t\sim\normallaw{R_t}$ is
  the observation noise, with respectively $Q_t$ and $R_t$ the covariance
  matrix. These two noises are independent between them, further, $W_t$ of different
  time instant are also mutually independent (the same for $V_t$).
\end{itemize}
Suppose that $X_0$ is Gaussian. Then it follows that the process
$(X_t,Y_t)_{t\geq 0}$ is Gaussian.  The objective is to estimate
the system state $X_t$ from the accumulated observations
$\Yt:=[Y_1\ldots Y_t]$.

The optimal estimator (in the least-squares sense) of the system
state $X_t$ given the observations $\Yt$ is the conditional
expectation
\begin{align}
  \label{eq:conditional_estimator}
  \xtt = \EE[{X_t|\Yt}].
\end{align}
Since the joint vector $(X_t,\Yt)$ is Gaussian, the conditional
expectation $\xtt$ is a linear combination of $\Yt$, which can be
written in terms of $\hat{x}_{t-1|t-1}$ and $Y_t$ only. The
purpose of the KF is to calculate $\xtt$ from $\hat{x}_{t-1|t-1}$
and $Y_t$.

%We denote by $\hat x_t$ an estimation of $X_t$. An optimal
%estimation would minimize the quadratic loss function: $\EE
%[{|{X_t-\hat x_t}|^2}]$. Using the fact that $\EE[{|{X_t-\hat
%x_t}|^2}] = {\EE[{|{X_t-\hat x_t}|^2 | \Yt}]}$, we easily see that
%the minimizer is attained at
%\begin{align}
%  \label{eq:conditional_exp_min}
%  \xtt = {\rm{argmin}}_{\hat x_t} \EE[{|{X_t-\hat x_t}|^2 | \Yt}]\;.
%\end{align}
%By taking the derivative in (\ref{eq:conditional_exp_min}) with
%respect to $\hat x_t$, it follows that the optimal estimator of
%the system state $X_t$ is given by
%\begin{align}
%  \label{eq:conditional_estimator}
%  \xtt = \EE[{X_t|\Yt}]\;.
%\end{align}
%The conditional distribution $X_t|\Yt$ being still Gaussian,
% the optimal estimator is also given by
%\begin{align}
%  \label{eq:MAP_estimator}
%  \xtt = {\rm{argmax}}_{\hat x_t} P(X_t=\hat x_t|\Yt)\;.
%\end{align}
%The KF calculates recursively $\xtt = \EE[{X_t|\Yt}]$. It turns
%out that $\xtt$ is unbiased, and linear on $\Yt$. It is optimal in
%the sense of \eqref{eq:conditional_exp_min} or
%\eqref{eq:MAP_estimator}.

We summarize the algorithm in the following.

\textbf{Initialization:}
\begin{align}
  \label{eq:Kalman_init}
  \hat x_{0|0} = \EE[{X_0}], \ P_{0|0} = \cov{X_0}.
\end{align}

\textbf{Prediction:}
\begin{align}
  \label{eq:Kalman_prediction}
  \xtn &= F_t \xnn ,\\
  \Ye_t &= Y_t - H_t \xtn ,\\
  P_{t|t-1} &= F_t P_{t-1|t-1} F_t^T + Q_t .
\end{align}

\textbf{Update:}
\begin{align}
  \label{eq:Kalman_update}
  S_t &= H_t P_{t|t-1} H_t^T + R_t ,\\
  K_t &= P_{t|t-1}H_t^T S_t^{-1} ,\\
  \xtt &= \xtn + K_t \Ye_t , \\
  P_{t|t} &= (I - K_tH_t) P_{t-1|t-1} .
\end{align}
To apply the KF algorithm the covariance matrices $Q_t, R_t$ must
be known.

\section{Extended Kalman Filter}
\label{sec:extend-kalm-filt}

Consider now a nonlinear dynamical system:
\begin{align}
  X_{t} &= f_t(X_{t-1},W_t) ,\\ %\label{eqn:nlsys_etat}
  Y_{t} &= h_t(X_t,V_t) ,\label{eqn:nlsys_obs}
\end{align}
where $X_t, Y_t, W_t, V_t$ are the same as in the KF, while the
functions $f_t, h_t$ are nonlinear and differentiable. Nothing can
be said in general on the conditional distribution $X_t|\Yt$ due
to the nonlinearity. The EKF calculates an approximation of the
conditional expectation \eqref{eq:conditional_estimator} by an
appropriate linearization of the state transition and observation
models, which makes the general scheme of KF still applicable
\cite{kalman2}. However, the resulting algorithm is no more
optimal in the least-squares sense due to the approximation.

Let $F_X=\partial_X f(\xnn, 0), F_W=\partial_W f(\xnn, 0)$, the
partial derivatives of $f$ (with respect to the system state and
the process noise) evaluated at $(\xnn, 0)$, and let
$H_X=\partial_X h(\xtn, 0), H_V=\partial_V h(\xtn, 0)$ be the
partial derivatives of $h$ (with respect to the system state and
the observation noise) evaluated at $(\xtn, 0)$. The EKF algorithm
is summarized below.

\textbf{Initialization:}
\begin{align}
  \label{eq:EKF_init}
  \hat x_{0|0} = \EE[{X_0}], \ P_{0|0} = \cov{X_0}.
\end{align}

\textbf{Prediction:}
\begin{align}
  \label{eq:EKF_prediction}
  \xtn &= f(\xnn,0) , \\
  \Ye_t &= Y_t - h(\xtn, 0) ,\\
  P_{t|t-1} &= F_X P_{t-1|t-1} F_X^T + F_WQ_tF_W^T .
\end{align}

\textbf{Update:}
\begin{align}
  \label{eq:EKF_update}
  S_t &= H_X P_{t|t-1} H_X^T + H_VR_tH_V^T , \\
  K_t &= P_{t|t-1}H_X^T S_t^{-1} ,\\
  \xtt &= \xtn + K_t \Ye_t , \\
  P_{t|t} &= (I - K_tH_X) P_{t-1|t-1} .
\end{align}

\bibliographystyle{plain}

\end{document}